\documentclass{article}
\usepackage{amssymb,amsfonts,amsmath,amsthm,amsopn,amstext,amscd,latexsym,xy,color,stmaryrd,epic}
\usepackage{hyperref}
\usepackage{tikz,multicol,graphicx,longtable}
\usetikzlibrary{arrows,positioning,decorations.pathmorphing,
  decorations.markings
 }
\tikzset{inner sep=0pt, 
  root/.style={circle,draw,minimum size=7pt,thick}, 
  fatroot/.style={circle,draw,minimum size=10pt,thick}, 
  short root/.style={circle,fill,minimum size=7pt}, 
  doublearrow/.style={postaction={decorate}, 
  decoration={markings,mark=at position .7
  with {\arrow{angle 60}}},double distance=3pt,thick}
} 

\theoremstyle{plain}
\input xy
\xyoption{all}
\newtheorem{theorem}{Theorem}[section]
\newtheorem{lemma}[theorem]{Lemma}
\newtheorem{proposition}[theorem]{Proposition}
\newtheorem{corollary}[theorem]{Corollary}
\newtheorem{definition}[theorem]{Definition}

\theoremstyle{remark}

\numberwithin{equation}{section}
\numberwithin{paragraph}{section}

\newcommand{\dquot}{{\,\!\sslash\!\,}}

\DeclareMathOperator{\loc}{loc}
\DeclareMathOperator{\Hom}{Hom}
\DeclareMathOperator{\vol}{vol}

\DeclareMathOperator{\Ad}{Ad}

\DeclareMathOperator{\rank}{rank}

\DeclareMathOperator{\Gal}{Gal}
\DeclareMathOperator{\Aut}{Aut}
\DeclareMathOperator{\Stab}{Stab}
\DeclareMathOperator{\Sel}{Sel}

\DeclareMathOperator{\im}{im}

\DeclareMathOperator{\Lie}{Lie}

\DeclareMathOperator{\Pic}{Pic}

\DeclareMathOperator{\Spec}{Spec}

\newcommand{\cB}{{\mathcal B}}
\newcommand{\cC}{{\mathcal C}}

\newcommand{\cF}{{\mathcal F}}

\newcommand{\cJ}{{\mathcal J}}

\newcommand{\cM}{{\mathcal M}}

\newcommand{\cO}{{\mathcal O}}

\newcommand{\cV}{{\mathcal V}}

\newcommand{\cX}{{\mathcal X}}
\newcommand{\cY}{{\mathcal Y}}
\newcommand{\cZ}{{\mathcal Z}}

\newcommand{\frc}{{\mathfrak c}}

\newcommand{\frg}{{\mathfrak g}}
\newcommand{\frh}{{\mathfrak h}}

\newcommand{\frl}{{\mathfrak l}}

\newcommand{\frs}{{\mathfrak s}}

\newcommand{\frt}{{\mathfrak t}}

\newcommand{\frz}{{\mathfrak z}}
\newcommand{\bbA}{{\mathbb A}}

\newcommand{\bbC}{{\mathbb C}}

\newcommand{\bbF}{{\mathbb F}}
\newcommand{\bbG}{{\mathbb G}}

\newcommand{\bbP}{{\mathbb P}}
\newcommand{\bbQ}{{\mathbb Q}}
\newcommand{\bbR}{{\mathbb R}}

\newcommand{\bbZ}{{\mathbb Z}}

\newcommand{\GL}{\mathrm{GL}}

\newcommand{\SL}{\mathrm{SL}}

\newcommand{\al}{\alpha}
\newcommand{\be}{\beta}
\newcommand{\del}{\delta}

\newcommand{\lam}{\lambda}
\newcommand{\varep}{\varepsilon}

\newcommand{\Siegel}{\mathfrak{S}}
\newcommand{\ZV}{V}

\newcommand{\WStab}{\Omega}
\newcommand{\htvar}{a}
\newcommand{\p}{c}

\DeclareMathOperator{\irr}{irr}
\DeclareMathOperator{\Ht}{ht}

\AtEndDocument{\bigskip{\footnotesize 
\textsc{Beth Romano}~~ \texttt{blr24@dpmms.cam.ac.uk}\\\textsc{Department of Pure Mathematics and Mathematical Physics, Wilberforce Road, Cambridge, CB3 0WB, UK}
\par
\textsc{Jack A. Thorne} ~~ \texttt{thorne@dpmms.cam.ac.uk}\\  \textsc{Department of Pure Mathematics and Mathematical Physics, Wilberforce Road, Cambridge, CB3 0WB, UK}}}

\title{On the arithmetic of simple singularities of type $E$}
\author{Beth Romano and Jack A. Thorne}

\begin{document}
\maketitle
\begin{abstract}
An ADE Dynkin diagram gives rise to a family of algebraic curves. In this paper, we use arithmetic invariant theory to study the integral points of the curves associated to the exceptional diagrams $E_6, E_7$, $E_8$. These curves are non-hyperelliptic of genus 3 or 4. We prove that a positive proportion of each family consists of curves with integral points everywhere locally but no integral points globally.
\end{abstract}
\tableofcontents
\section{Introduction}

\paragraph{Background.} 
Consider the following families of affine plane curves over $\bbQ$:
\begin{equation}\label{eqn_intro_curves_of_type_E6}
y^3 = x^4 + y(\p_2 x^2 + \p_5 x + \p_8) + \p_6 x^2 + \p_9 x + \p_{12}
\end{equation}
\begin{equation}\label{eqn_intro_curves_of_type_E7}
y^3 = x^3 y + \p_{10} x^2 + x(\p_2 y^2 + \p_8 y + \p_{14} ) + \p_6 y^2 + \p_{12} y + \p_{18}
\end{equation}
\begin{equation}\label{eqn_intro_curves_of_type_E8}
y^3 = x^5 + y(\p_2 x^3 + \p_8 x^2 + \p_{14} x + \p_{20} ) + \p_{12} x^3 + \p_{18} x^2 + \p_{24} x + \p_{30}.
\end{equation}
These families arise as versal deformations of the simple plane singularities of types $E_6$, $E_7$, and $E_8$, respectively (see \cite{Tho13}). In each family, the singularity can be recovered by setting all coefficients $\p_i$ equal to 0; yet the generic member of each family is smooth, and its smooth projective completion acquires rational points at infinity. Thus it is natural to study the arithmetic of these families of pointed smooth projective curves. The study of these families can be viewed as a variation on a classical theme: if we started instead with the singularity of type $A_2$ (given by the equation $y^2 = x^3$), then we would be studying the arithmetic of elliptic curves in standard Weierstrass form.

We recall that if $Y$ is a smooth projective curve over a global field $k$ and $P \in Y(k)$ is a rational point, then one can define the 2-Selmer set $\Sel_2 Y$ of the curve $Y$; it is a subset of the 2-Selmer group of the Jacobian of $Y$ that serves as a cohomological proxy for the set $Y(k)$ of $k$-rational points. In the paper \cite{Tho15}, the second author studied the behaviour of the 2-Selmer sets of the curves in the family (\ref{eqn_intro_curves_of_type_E6}), proving the following theorem (\cite[Theorem 4.3]{Tho15}):
\begin{theorem}\label{thm_intro_E6_small_selmer}
Let $\cF_0 \subset \bbZ^6$ denote the set of tuples $(\p_2, \p_5, \p_8, \p_6, \p_9, \p_{12}) \in \bbZ^6$ such the the affine curve given by equation (\ref{eqn_intro_curves_of_type_E6}) is smooth (over $\bbQ$). If $b \in \cF_0$, then call $\Ht(b) = \sup_i | \p_i(b) |^{72/i}$ the height of $b$, and let $Y_b$ denote the smooth projective completion of the fibre $X_b$ as an algebraic curve over $\bbQ$. If $\cF \subset \cF_0$ is a subset defined by congruence conditions, then we have
\[ \limsup_{\htvar \to \infty} \frac{\sum_{\substack{b \in \cF \\ \Ht(b) < \htvar}} \# \Sel_2(Y_b) }{\#\{ b \in \cF \mid \Ht(b) < \htvar \}}  < \infty. \]
Moreover, for any $\epsilon > 0$, we can find a subset $\cF \subset \cF_0$ defined by congruence conditions such that
\[ 1 \leq \limsup_{\htvar \to \infty} \frac{\sum_{\substack{b \in \cF \\ \Ht(b) < \htvar}} \# \Sel_2(Y_b) }{\#\{ b \in \cF \mid \Ht(b) < \htvar \}}  < 1 + \epsilon. \]
\end{theorem}
For the definition of a subset defined by congruence conditions, see (\ref{eqn_intro_defined_by_congruence_conditions}) below. This theorem has the following Diophantine consequence (\cite[Theorem 4.8]{Tho15}): 
\begin{theorem}\label{thm_intro_E6_few_integral_points}
Let $\epsilon > 0$, and let $\cF_0$ be as in the statement of Theorem \ref{thm_intro_E6_few_integral_points}. If $b \in \cF_0$, let $\cX_b$ denote the affine curve over $\bbZ$ given by the equation (\ref{eqn_intro_curves_of_type_E6}). Then there exists a subset $\cF \subset \cF_0$ defined by congruence conditions that satisfies the following conditions:
 \begin{enumerate}
 \item For every $b \in \cF$ and for every prime $p$, $\cX_b(\bbZ_p) \neq \emptyset$.
 \item We have
 \[ \liminf_{\htvar \to \infty} \frac{\# \{ b \in \cF \mid \Ht(b) < \htvar,\text{ } \cX_b(\bbZ) = \emptyset \} }{\#\{ b \in \cF \mid \Ht(b) < \htvar \}} > 1 - \epsilon. \]
 \end{enumerate}
\end{theorem}
In other words, a positive proportion of curves in the family (\ref{eqn_intro_curves_of_type_E6}) have no $\bbZ$-points despite having $\bbZ_p$-points for every prime $p$. (The presence of marked points at infinity implies that for every $b \in \cF_0$, the curve $\cX_b$ also has $\bbR$-points.)

\paragraph{The results of this paper.} The goal of this paper is to generalize these results to the other two families (\ref{eqn_intro_curves_of_type_E7}) and (\ref{eqn_intro_curves_of_type_E8}) described above. The techniques we use are broadly similar to those of \cite{Tho15}, and are based around the relation, introduced in \cite{Tho13}, between the arithmetic of these families of curves and certain Vinberg representations associated to the corresponding root systems. We study this relation and then employ the orbit-counting techniques of Bhargava to prove our main theorems. We refer the reader to \cite[Introduction]{Tho15} for a more detailed discussion of these ideas. 

In order to state the main theorems of this paper precisely, we must introduce some more notation. We will find it convenient to state our results in parallel for the two families (\ref{eqn_intro_curves_of_type_E7}) and (\ref{eqn_intro_curves_of_type_E8}). When it is necessary to split into cases, we will say that we are either in Case $\mathbf{E_7}$ or in Case $\mathbf{E_8}$. We specify the following notation:
\begin{itemize}
\item[Case $\mathbf{E_7}$:] We let $\cB$ denote the affine scheme $\bbA^7_\bbZ$ with coordinates $(\p_2, \p_6, \p_8, \p_{10}, \p_{12}, \p_{14}, \p_{18})$, and let $B = \cB_\bbQ$. We let $\cX \subset \bbA^2_\cB$ denote the affine curve over $\cB$ given by the equation (\ref{eqn_intro_curves_of_type_E7}), and $X = \cX_\bbQ$. We let $Y \to B$ denote the family of projective curves defined in \cite[Lemma 4.9]{Tho13} (this family is a fibre-wise compactification of $X$ that is smooth at infinity. It can be realized as the closure of $X$ in $\bbP^2_B$). We let $\cF_0$ denote the set of $b \in \cB(\bbZ)$ such that $X_b$ is smooth. If $b \in \cF_0$, then we define $\Ht(b) = \sup_i | \p_i(b) |^{126/i}$.
\item[Case $\mathbf{E_8}$:] We let $\cB$ denote the affine scheme $\bbA^8_\bbZ$ with coordinates $(\p_2, \p_8, \p_{12}, \p_{14}, \p_{18}, \p_{20}, \p_{24}, \p_{30})$, and let $B = \cB_\bbQ$. We let $\cX \subset \bbA^2_\cB$ denote the affine curve over $\cB$ given by the equation (\ref{eqn_intro_curves_of_type_E8}), and $X = \cX_\bbQ$. We let $Y \to B$ denote the family of projective curves defined in \cite[Lemma 4.9]{Tho13} (again, this family is a fibre-wise compactification of $X$ that is smooth at infinity. It can be realized as the closure of $X$ in a suitable weighted projective space over $B$). We let $\cF_0$ denote the set of $b \in \cB(\bbZ)$ such that $X_b$ is smooth. If $b \in \cF_0$, then we define $\Ht(b) = \sup_i | \p_i(b) |^{240/i}$.
\end{itemize}
In either case, we say that a subset $\cF \subset \cF_0$ is defined by congruence conditions if there exist distinct primes $p_1, \dots, p_s$ and a non-empty open compact subset $U_{p_i} \subset \cB(\bbZ_{p_i})$ for each $i \in \{1, \dots, s\}$ such that
\begin{equation}\label{eqn_intro_defined_by_congruence_conditions} \cF = \cF_0 \cap (U_{p_1} \times \dots \times U_{p_s}), 
\end{equation}
where we are taking the intersection inside $\cB(\bbZ_{p_1}) \times \dots \times \cB(\bbZ_{p_s})$.
Our first main result is then as follows.
\begin{theorem}\label{thm_intro_small_selmer}
\begin{enumerate}
\item Let $\cF_0 \subset \cF$ be a subset defined by congruence conditions. Then we have
\[ \limsup_{\htvar \to \infty} \frac{\sum_{\substack{b \in \cF \\ \Ht(b) < \htvar}} \# \Sel_2(Y_b) }{\#\{ b \in \cF \mid \Ht(b) < \htvar \}}  < \infty. \]
\item For any $\epsilon > 0$, we can find a subset $\cF \subset \cF_0$ defined by congruence conditions such that
\[ \limsup_{\htvar \to \infty} \frac{\sum_{\substack{b \in \cF \\ \Ht(b) < \htvar}} \# \Sel_2(Y_b) }{\#\{ b \in \cF \mid \Ht(b) < \htvar \}}  < \left\{ \begin{array}{cc} 2 + \epsilon & \text{\emph{Case} }\mathbf{E_7}; \\ 1 + \epsilon & \text{\emph{Case} }\mathbf{E_8}. \end{array} \right. \]
\end{enumerate}
\end{theorem}
(We note that the average in Case $\mathbf{E_7}$ is at least 2, because the family of curves (\ref{eqn_intro_curves_of_type_E7}) has two marked points at infinity; for a generic member of this family, these rational points define distinct elements inside the 2-Selmer set $\Sel_2 Y_b$). In either case, we can apply Theorem \ref{thm_intro_small_selmer} to deduce the following consequence.
\begin{theorem}\label{thm_intro_few_integral_points}
 Let $\epsilon > 0$. Then there exists a subset $\cF \subset \cF_0$ defined by congruence conditions satisfying the following conditions:
 \begin{enumerate}
 \item For every $b \in \cF$ and for every prime $p$, $\cX_b(\bbZ_p) \neq \emptyset$.
 \item We have
 \[ \liminf_{\htvar \to \infty} \frac{\# \{ b \in \cF \mid \Ht(b) < \htvar,\text{ } \cX_b(\bbZ) = \emptyset \} }{\#\{ b \in \cF \mid \Ht(b) < \htvar \}} > 1 - \epsilon. \]
 \end{enumerate}
\end{theorem}
Informally, we have shown that a positive proportion of each of the families (\ref{eqn_intro_curves_of_type_E7}) and (\ref{eqn_intro_curves_of_type_E8}) consists of curves with $\bbZ_p$-points for every prime $p$ but no $\bbZ$-points.

\paragraph{Methodology.} We now describe some new aspects of the proofs of Theorem \ref{thm_intro_small_selmer} and Theorem \ref{thm_intro_few_integral_points}. The main steps of our proofs are the same as those of \cite{Tho15}: we combine the parameterization (constructed in \cite{Tho13}) of 2-Selmer elements by rational orbits in a certain representation $(G, V)$ arising from a graded Lie algebra with a technique of counting integral orbits (i.e. of the group $G(\bbZ)$ in the set $V(\bbZ)$). We thus gain information about the average size of 2-Selmer sets.

Although our proofs are similar in outline to those of \cite{Tho15}, we need to introduce several new ideas here. For example, the most challenging technical step in the argument is to eliminate the contribution of integral points which lie `in the cusp'. (In the notation of Section \ref{section-reducible}, these points correspond to vectors $v$ such that $v_{\al_0} = 0$, where $\al_0$ is the highest root in the ambient Lie algebra $\frh$.) For this step we prove an optimized criterion (Proposition \ref{prop_reducibility_conditions}) for when certain vectors are reducible (this implies that they cannot contribute to the nontrivial part of the 2-Selmer set of a smooth curve in our family).This criterion is based in large part on the Hilbert--Mumford stability criterion. Its application in this context is very natural, but seems to be new. 

We then use a computer to carry out a formidable computation to bound the contribution of the parts of the cuspidal region that are not eliminated by this criterion (see Proposition \ref{prop-cuspdatum}). For comparison, we note that in \cite{Tho15}, the cuspidal region was broken up into 68 pieces; here the analogous procedure leads to a decomposition into $1429$ (resp. $9437$ pieces) in Case $\mathbf{E_7}$ (resp. in Case $\mathbf{E_8}$). It would be very interesting if one could discover a `pure thought' way to tackle this problem that does not rely on case-by-case calculations.

The current setting also differs from that of \cite{Tho15} in that the curves of family (\ref{eqn_intro_curves_of_type_E7}) have more than one marked point at infinity. (The geometric reason for this is that the projective tangent line to a flex point $P$ of a plane quartic curve intersects the curve in exactly one other point $Q$. This implies that the family (\ref{eqn_intro_curves_of_type_E7}), essentially the universal family of plane quartics with a marked flex point, has two canonical sections.) We find that the orbits that parameterize the divisor classes arising from these points match up in a very pleasant way with a certain subgroup of the Weyl group of the ambient Lie algebra $\frh$; see in particular Lemma \ref{lem_weyl_group_represents_component_group}.

It remains an interesting open problem to generalize the results of this paper and of \cite{Tho15} to study the average size of the 2-Selmer group of the Jacobians of the curves in (\ref{eqn_intro_curves_of_type_E6}) -- (\ref{eqn_intro_curves_of_type_E8}) (and not just the size of their 2-Selmer sets). The rational orbits necessary for this study were constructed in \cite{Tho16}, but we do not yet understand how to construct integral representatives for these orbits, in other words, how to prove the analogue of Lemma \ref{lem_existence_of_local_integral_orbits} below after replacing the set $Y_b(\bbQ_p)$ by $J_b(\bbQ_p)$. If this can be achieved, then the work we do in this paper to bound the contribution of the cuspidal region will suffice to obtain the expected upper bound on the average size of the 2-Selmer group (namely 6 in Case $\mathbf{E_7}$ and 3 in Case $\mathbf{E_8}$).

\paragraph{Notation.} Given a connected reductive group $H$ and a maximal torus $T \subset H$, we write $X^*(T) = \Hom (T, \bbG_m)$ for the character group of $T$, $X_*(T)$ for the cocharacter group of $T$, and $W(H, T)$ for the (absolute) Weyl group of $H$ with respect to $T$. Similarly, if $\frc$ is a Cartan subalgebra of $\frh = \Lie(H)$, then we write $\Phi(\frh, \frc)$ for the roots of $\frc$ and $W(H, \frc)$ for the Weyl group of $\frc$.  If $\al \in \Phi(\frh, \frc)$, then we write $\frh_\al \subset \frh$ for the root space corresponding to $\al$. We write $N_H(T)$ (resp. $N_H(\frc)$) for the normalizer of $T$ (resp. $\frc$) in $H$, and $Z_H(T)$ (resp. $Z_H(\frc)$) for the associated centralizer. Similarly, if $V$ is any subspace of $\frh$ and $x \in \frh$, then we write $\frz_V(x)$ for the centralizer of $x$ in $V$.

We write $\Lambda = \bbR_{>0}$ for the multiplicative group of positive reals, and $d^\times \lambda = d \lambda / \lambda$ for its Haar measure (where $d \lambda$ is the usual Lebesgue measure on the real line).
If $G$ is a group defined over a ring $R$, $V$ is an representation of $G$, and $A \subset V$, then we write $G(R) \backslash A$ for the set of equivalence classes of $A$ under the relation $a \sim a'$ if there exists $\gamma \in G(R)$ such that $\gamma a = a'$.
\paragraph{Acknowledgements.} During the period in which this research was conducted, Jack Thorne served as a Clay Research Fellow. Both of the authors were supported in part by EPSRC First Grant EP/N007204/1. We thank Fabrizio Barroero for useful conversations. 

\section{A stable grading}\label{sec_a_stable_grading}

In this section we establish the algebraic foundation for the proofs of our main theorems: in each of our two cases, we describe the parameterization of certain 2-coverings of Jacobians of algebraic curves by orbits in a representation arising from a $\bbZ / 2 \bbZ$-graded Lie algebra. Our set-up parallels that of \cite{Tho15}; however, we must address the complications arising from the presence of an additional point at infinity on the curves in the family (\ref{eqn_intro_curves_of_type_E7}). This point makes its presence known in 
the disconnectedness of the group $H^\theta$ defined below 
and in the fact that the central fibre of the family (\ref{eqn_intro_curves_of_type_E7}) is not irreducible. 

\subsection{Definition of the grading}\label{sec_vinberg_representations}

Let $k$ be a field of characteristic 0 with fixed separable closure $k^s$, and let $H$ be a simple adjoint group over $k$ of rank $r$ that is equipped with a $k$-split maximal torus $T$. Let $\frh = \Lie(\frh)$ and $\frt = \Lie(\frt)$. We let $\Phi_H = \Phi(H, T)$ and choose a set of simple roots $S_H = \{\al_1, \al_2, ..., \al_r \} \subset \Phi_H$. We also choose a Chevalley basis for $\frh$ with root vectors $\{e_\al \mid \al \in \Phi_H\}$. Suppose that $-1$ is an element of the Weyl group $W(H, T)$ (this is true, e.g., if $H$ has type $E_7$ or $E_8$, but not if $H$ has type $E_6$). Let $\check{\rho} \in X_\ast(T)$ be the sum of the fundamental coweights with respect to our choice of simple roots $S_H$. Then, up to conjugation by $H(k)$, the automorphism $\theta := \Ad(\check\rho(-1))$ is the unique involution of $H$ such that $\frh^{d\theta = -1}$ contains a regular nilpotent element of $\frh$ (\cite[Corollary 2.15]{Tho13}). 
The grading induced by this involution is stable in the sense of \cite[\S 5.3]{RLYG}.

We define $G = (H^\theta)^\circ$ and $V = \frh^{d \theta = -1}$. Then $G$ is a split semisimple group, and $V$ is an irreducible representation of $G$, of the type studied by Kostant--Rallis in the case $k = \bbC$ \cite{Kos71}. The invariant theory of $V$ is closely related to that of the adjoint representation of $H$. We now summarize some aspects of the invariant theory of the pair $(G, V)$, most of which may be found in \cite{Kos71}, \cite{Vin76}, or \cite{Pan05}. We refer the reader to \cite[\S 2]{Tho13} for detailed references. 
\begin{definition}
Let $\frc \subset \frh$ be a Cartan subalgebra. If $\frc \subset V$, then $\frc$ is called a \emph{Cartan subspace} of $V$.
\end{definition}

\begin{theorem}
\begin{enumerate}
\item Any two Cartan subspaces $\frc, \frc' \subset V$ are conjugate by an element of $G(k^s)$.
\item Let $\frc \subset V$ be a Cartan subspace, and define $W(G, \frc) = N_G(\frc) / Z_G(\frc)$. Then the natural maps
\[ W(G, \frc) \to W(H, \frc) \]
and
\[ k[\frh]^H  \to k[V]^G  \to k[\frc]^{W(G,\frc)} \]
are isomorphisms. In particular, $k[V]^G$ is isomorphic to a polynomial algebra on $r = \rank H$ generators. 
\end{enumerate}
\end{theorem}
Let us call a vector $v \in V$ semisimple (resp. nilpotent, resp. regular) if it has this property when viewed as an element of $\frh$. We have the following proposition:
\begin{proposition}
Let $v \in V$.
\begin{enumerate}
\item The components of the Jordan decomposition $v = v_s + v_n$ in $\frh$ in fact lie in $V$.
\item The vector $v$ has a closed $G$-orbit in $V$ if and only if it is semisimple.
\item The stabilizer of $v$ in $G$ is finite (and hence the $G$-orbit of $v$ has maximal dimension) if and only if $v$ is regular.
\end{enumerate}
\end{proposition}

We see in particular that a vector $v \in V$ has both a closed orbit and a finite stabilizer (i.e. $v$ is stable in the sense of \cite{Mumford}) if and only if it is regular semisimple. Let $\tilde\Delta \in k[\frh]^H$ be the image under the isomorphism $k[\frt]^{W(H, T)} \to k[\frh]^{H}$ of the product of all roots $\alpha \in \Phi_H$. Then $\tilde\Delta(v) \neq 0$ if and only if $v \in \frh$ is regular semisimple. We call $\Delta := \tilde\Delta|_V$ the discriminant polynomial. Then $\Delta$
is homogeneous of degree $\# \Phi_H$.
If $v \in V$ is a vector such that $\Delta(v) \neq 0$, then $\frz_\frh(v) \subset V$, and $\frz_\frh(v)$ is the unique Cartan subspace of $V$ containing $v$. 

Before stating the next result, we review some basic definitions from geometric invariant theory. Recall that given a one-parameter subgroup $\lam: \mathbb{G}_m \to G_{k^s}$, we may decompose $V(k^s)$ as $\oplus_{i \in \bbZ} V_i$, where $V_i = \{v \in V(k^s) \mid \lam(t)\cdot v = t^i v\}$. If we decompose a vector $v \in V$ as $v = \sum v_i$ where $v_i \in V_i$ for all $i$, then $\{ i \mid v_i \neq 0 \}$ is called the set of weights for $v$ with respect to $\lam$.

\begin{corollary}\label{cor_hilbert_mumford_criterion}
Let $v \in V$. Then the following are equivalent:
\begin{enumerate}
\item $v$ is regular semisimple.
\item $\Delta(v) \neq 0$.
\item For any nontrivial one-parameter subgroup $\lam: \bbG_m \to G_{k^s}$, the vector $v$ has a positive weight with respect to $\lam$.
\end{enumerate}
\end{corollary}
\begin{proof}
What remains to be shown is that the third condition is equivalent to the vector $v$ having a closed orbit and a finite stabilizer in $G$. This is the Hilbert--Mumford stability criterion (see e.g. \cite{Mumford}).
\end{proof}

We now describe $G$ and $V$ more explicitly. By our definition of $\theta$, it is clear that $T \subset G$. Let $\Phi_G = \Phi(G, T)$; then $\Phi_G \subset \Phi_H$, and the complement $\Phi_V := \Phi_H - \Phi_G$ is the set of weights for the action of $T$ on $V$. The Weyl group $W_G := W(G, T)$ is the subgroup of $W_H := W(H, T)$ generated by reflections corresponding to the roots of $\Phi_G$. 

\begin{lemma}\label{lem_weyl_group_represents_component_group}
Let $s = \check\rho(-1) \in T(k)$.
\begin{enumerate}
\item The stabilizer of $s$ under the action of $W_H$ on $T$ is given by $\Stab_{W_H}(s) = \{ w \in W_H \mid w(\Phi_G) = \Phi_G \}$.
\item There is a split short exact sequence of groups 
\[ \xymatrix@1{1 \ar[r] & W_G \ar[r] & \Stab_{W_H}(s)  \ar[r] & H^\theta / G \ar[r] & 1.} \]
More precisely, let $S_G \subset \Phi_G$ be a choice of root basis and define 
\[ \WStab = \{ w \in W_H \mid w(S_G) = S_G \} \subset \Stab_{W_H}(s). \]
 Then $\Stab_{W_H}(s) \simeq W_G \rtimes \WStab$, and the inclusion $N_{H^\theta}(T) \hookrightarrow H^\theta$ induces an isomorphism $\WStab \simeq H^\theta / G$.
\end{enumerate}
\end{lemma}
We remark that if $H$ is of type $E_7$, then the group $H^\theta / G$ has order 2; if $H$ is of type $E_8$, then $H^\theta/G$ is trivial.
\begin{proof} 
For the first item, note that since $H$ is adjoint, $w\cdot s$ is completely determined by its action on the root spaces $\frh_\al$.
We have that $w\cdot s$ acts trivially on $\frh_\al$ if and only if $\al \in w^{-1}(\Phi_G)$, and otherwise $w\cdot s$ acts on $\frh_\al$ as multiplication by $-1$.
For the second item, note that by item 1, the group $\Stab_{W_H}(s)$ is a subgroup of $\Aut(\Phi_G) \simeq W_G \rtimes D$, where $D = \{\sigma \in \Aut(\Phi_G) \mid \sigma(S_G) = S_G\}$. Clearly $W_G \subset \Stab_{W_H}(s)$ and $\Stab_{W_H}(s) \cap D = \WStab$, so $\Stab_{W_H}(s) \simeq W_G \rtimes \WStab$.
The isomorphism with $H^\theta/G$ follows from \cite[Section 2.2]{Humphreys}. 
\end{proof}

\subsection{Transverse slices over $V \dquot G$}\label{section_transverse_slices}

We continue to use the notation of \S \ref{sec_vinberg_representations}, and now begin our study of the categorical quotient map
\[ \pi : V \to B, \]
where $B = V \dquot G = \Spec k[V]^G$. If $b \in B(k)$, we write $V_b = \pi^{-1}(b)$ for the corresponding fibre. We can write down sections of the map $\pi$ using the theory of $\frs\frl_2$-triples. We recall that an $\frs\frl_2$-triple in $\frh$ is a tuple $(e, h, f)$ of elements of $\frh - \{ 0 \}$ satisfying the relations
\begin{equation*}
[h, e] = 2e,\, [h, f] = -2f , \, [e, f] = h.
\end{equation*}
We call an $\frs\frl_2$-triple normal if $e,f \in V$ and $h \in \frh^\theta$. A graded version of the Jacobson--Morozov theorem (\cite[Lemma 2.17]{Tho13}) states that if $e \in V$ is a non-zero nilpotent element, then there exists a normal $\frs\frl_2$-triple containing it. If $(e, h, f)$ is a normal $\frs\frl_2$-triple, then we define $S_{(e, h, f)} = e + \frz_\frh(f) \cap V \subset V$. Then $S_{(e, h, f)}$ is an affine linear subspace containing $e$, and one can show (\cite[Proposition 3.4]{Tho13}) that the map $\pi|_{S_{(e, h, f)}} : S_{(e, h, f)} \to B$ is faithfully flat, with smooth generic fibre. If we let $\lam:\bbG_m \to H$ be the cocharacter such that $d\lam(1) = h$, then we may define a contracting action of $\bbG_m$ on $S_{(e, h, f)}$ by $t\cdot v = t^2\lam(t^{-1})v$. With this action on $S_{(e, h, f)}$, if $\bbG_m$ acts on $B$ by the square of its usual action, then $\pi|_{S_{(e, h, f)}}$ is $\bbG_m$-equivariant (see \cite[\S 3]{Tho13}). If $e$ is regular nilpotent, then we call $S_{(e, h, f)}$ a Kostant section.\footnote{We note that the definition of a Kostant section is often more general than the one stated here, but in this paper we restrict our attention to sections of this form.}

We consider these affine subspaces for the $\frs\frl_2$-triples corresponding to two conjugacy classes of nilpotent elements, namely the regular and subregular classes. 
\begin{proposition}\label{prop_parameterization_of_orbits}
Let $E \in V$ be a regular nilpotent element. Then:
\begin{enumerate} \item There exists a unique normal $\frs\frl_2$-triple containing $E$.
Let $\kappa$ be the Kostant section associated to this $\frs\frl_2$-triple. Then $\pi|_\kappa$ is an isomorphism. 
\item Let $b \in B(k)$, and let $\kappa_b = (\pi|_\kappa)^{-1}(b)$. If $\Delta(b) \neq 0$, then $V_b$ forms a single $G(k^s)$-orbit. Consequently, there is a canonical bijection
\[ G(k) \backslash V_b(k) \cong \ker[ H^1(k, Z_G(\kappa_b)) \to H^1(k, G) ], \]
where the $G(k)$-orbit of $\kappa_b \in V_b(k)$ corresponds to the neutral element of $H^1(k, Z_G(\kappa_b))$.
\end{enumerate}
\end{proposition}
\begin{proof}
The first part follows from work of Kostant and Rallis as applied in \cite{Tho13}: see especially lemmas 2.17 and 3.5. 
 The second part follows from \cite[Proposition 1]{BhaGross14} as applied in \cite[Proposition 4.13]{Tho13}.
\end{proof}
For $b \in B(k)$, we continue to write $\kappa_b$ for the fibre over $b$. We observe that if $H$ has type $E_7$, then there are two $G$-conjugacy classes of regular nilpotent elements in $V$. If $H$ has type $E_8$, then there is a single $G$-conjugacy class of regular nilpotent elements (see \cite[Corollary 2.25]{Tho13}). In either case, two regular nilpotent elements $E, E' \in V(k)$ are $G(k)$-conjugate if and only if they are $G(k^s)$-conjugate (see e.g.\ \cite[Lemma 2.14]{Tho13}). Combined with the first part of Proposition \ref{prop_parameterization_of_orbits}, this implies a strong uniqueness property for the sections $\kappa \to B$:
\begin{corollary}\label{cor_Kostant_conj}
Let $\kappa, \kappa' \subset V$ be Kostant sections.
\begin{enumerate}
\item We have $\kappa = \kappa'$ if and only if $\kappa_0 = \kappa'_0$.
\item The sections $\kappa$ and $\kappa'$ are $G(k)$-conjugate if and only if $\kappa_0$ and $\kappa'_0$ lie in the same $G(k^s)$-orbit in $V$.
\end{enumerate}
\end{corollary}
Next recall that $V$ contains a subregular nilpotent element $e$ (by definition, this means that $e$ is nilpotent and $\dim \Stab_G(e) = 1$; the existence of subregular nilpotents in $V$ is proved in \cite[Proposition 2.27]{Tho13}). We now discuss the sections corresponding to such an element. 
\begin{theorem}\label{thm_identification_of_stabilizer}
Let $(e, h, f)$ be a normal $\frs\frl_2$-triple, and suppose that $e$ is subregular nilpotent element of $\frh$. Let $X = S_{(e, h, f)}$. 
\begin{enumerate} \item The fibres of $X \to B$ are reduced connected affine curves. If $b \in B(k)$, then $X_b$ is smooth if and only if $\Delta(b) \neq 0$.
\item Let $b \in B(k)$, and suppose that $\Delta(b) \neq 0$. Let $Y_b$ denote the smooth projective completion of $X_b$, and let $J_b = \Pic^0 Y_b$ be the Jacobian of $Y_b$. There is a canonical isomorphism $J_b[2] \cong Z_{G}(\kappa_b)$ of finite \'etale $k$-groups, where $\kappa$ is any choice of Kostant section. 
\end{enumerate}
\end{theorem}
\begin{proof}
For the first part, see \cite[Theorem 3.8]{Tho13} and \cite[Corollary 3.16]{Tho13}. For the second part, see \cite[Corollary 4.12]{Tho13}.
\end{proof}
 The next two theorems identify the fibres of the morphism $X \to B$ in Theorem \ref{thm_identification_of_stabilizer} when $H$ has type $E_7$ or $E_8$. We find it convenient to split into cases.
\begin{theorem}[Case $\mathbf{E_7}$]\label{thm_representation-theoretic_descent_diagram_case_E7}
Suppose that $H$ is of type $E_7$. Fix a choice of regular nilpotent $E$, and define $\kappa$ as in Proposition \ref{prop_parameterization_of_orbits}. Also fix a normal $\frs\frl_2$-triple $(e, h, f)$ such that $e$ is subregular nilpotent, and define $X = S_{(e, h, f)}$ as above. 
\begin{enumerate}
\item We may choose homogeneous generators $\p_2, \p_6, \p_8, \p_{10}, \p_{12}, \p_{14}, \p_{18}$ of $k[V]^G$ and functions $x, y \in k[X]$ so that $k[X]$ is isomorphic to a polynomial ring in the elements $\p_2, \dots, \p_{14}, x, y$, and the morphism $X \to B$ is determined by the relation (\ref{eqn_intro_curves_of_type_E7}):
\begin{equation*} y^3 = x^3 y + \p_{10} x^2 + x(\p_2 y^2 + \p_8 y + \p_{14} ) + \p_6 y^2 + \p_{12} y + \p_{18}. 
\end{equation*}
Moreover, the elements $\p_2, \p_6, \p_8, \p_{10}, \p_{12}, \p_{14}, \p_{18}, x, y \in k[X]$ are eigenvectors for the action of $\bbG_m$ on $X$ mentioned above, with weights as in the following table:
\begin{center}
\begin{tabular}{lllllll|ll}
$\p_2$ & $\p_6$ & $\p_8$ & $\p_{10}$ & $\p_{12}$ & $\p_{14}$ & $\p_{18}$ & $x$ & $y$  \\
\hline
$4$     & $12$   & $16$    & $20 $      & $24$       & $28$       & $36$       & $8$ & $12$
\end{tabular}
\end{center}
\item Let $Y \to B$ denote the natural compactification of $X \to B$ as a family of plane quartic curves, given in homogeneous coordinates as 
\[ y_0^3 z_0 = x_0^3 y_0 + \p_{10} x_0^2 z_0^2 + x_0(\p_2 y_0^2 z_0 + \p_8 y_0 z_0^2 + \p_{14} z_0^3) + \p_6 y_0^2 z_0^2 + \p_{12} y_0 z_0^3 + \p_{18}z_0^4. \]
This compactification has two sections $P_1$ and $P_2$ at infinity, given by the equations $[x_0 : y_0 : z_0] = [0 : 1 : 0]$ and $[x_0 : y_0 : z_0] = [1 : 0 : 0]$ respectively (note that $P_1$ is a flex point). Assume that under the bijection of \cite[Lemma 4.14]{Tho13} the section corresponding to $E$ is $P_1$.
Then for each $b \in B(k)$ such that $\Delta(b) \neq 0$, the following diagram commutes:
\[ \xymatrix{ X_b(k) \ar[r]^-{\iota_b} \ar[d]_{\eta_b} & G(k) \backslash V_b(k) \ar[d]^{\gamma_b} \\
J_b(k) \ar[r]^-{\del_b} & H^1(k, J_b[2]), } \]
where the maps in the diagram are specified as follows. The top arrow $\iota_b$ is induced by the inclusion $X \hookrightarrow V$. The left arrow $\eta_b$ is the restriction of the Abel--Jacobi map $P \mapsto [(P) - (P_1)]$. To define $\gamma_b$, we use Proposition \ref{prop_parameterization_of_orbits} to obtain an injective homomorphism to $G(k) \backslash V_b(k) \to H^1(k, Z_G(\kappa_b))$, and then compose with the identification $Z_G(\kappa_b) \cong J_b[2]$ of Theorem \ref{thm_identification_of_stabilizer}. The bottom arrow $\del_b$ is the connecting homomorphism associated to the Kummer exact sequence 
\[ \xymatrix@1{ 0 \ar[r] & J_b[2] \ar[r] & J_b \ar[r]^-{\times 2} & J_b \ar[r] & 0.} \]
\end{enumerate}
\end{theorem}
\begin{proof}
In this theorem and the next, the first part (i.e.\ the explicit determination of the family $X$) is carried out in \cite[Theorem 3.8]{Tho13}, the weights for the $\bbG_m$ action are given in \cite[Proposition 3.6]{Tho13}, and the second part is the content of \cite[Theorem 4.15]{Tho13}. 
\end{proof}
We note that, having fixed a choice of regular nilpotent $E$, we can always assume, after possibly replacing $e$ by a $H^\theta(k)$-conjugate, that $E$ corresponds to $P_1$ under the bijection of \cite[Lemma 4.14]{Tho13} referred to in the second part of Theorem \ref{thm_representation-theoretic_descent_diagram_case_E7}.
\begin{theorem}[Case $\mathbf{E_8}$]\label{thm_representation-theoretic_descent_diagram_case_E8}
Suppose that $H$ is of type $E_8$.  Fix a choice of regular nilpotent $E$, and define $\kappa$ as in Proposition \ref{prop_parameterization_of_orbits}. Also fix a normal $\frs\frl_2$-triple $(e, h, f)$ such that $e$ is subregular nilpotent, and define $X = S_{(e, h, f)}$ as above. 
\begin{enumerate}
\item We may choose homogeneous generators $\p_2, \p_8, \p_{12}, \p_{14}, \p_{18}, \p_{20}, \p_{24}, \p_{30}$ of $k[V]^G$ and functions $x, y \in k[X]$ so that $k[X]$ is isomorphic to a polynomial ring in the elements $\p_2, \dots, \p_{24}, x, y$, and the morphism $X \to B$ is determined by the relation (\ref{eqn_intro_curves_of_type_E8}):
\begin{equation*} y^3 = x^5 + y(\p_2 x^3 + \p_8 x^2 + \p_{14} x + \p_{20} ) + \p_{12} x^3 + \p_{18} x^2 + \p_{24} x + \p_{30}. 
\end{equation*}
Moreover, the elements $\p_2, \p_8, \p_{12}, \p_{14}, \p_{18}, \p_{20}, \p_{24}, \p_{30}, x, y \in k[X]$ are eigenvectors for the action of $\bbG_m$ on $X$ mentioned above, with weights as in the following table:
\begin{center}
\begin{tabular}{llllllll|ll}
$\p_2$ & $\p_8$ & $\p_{12}$ & $\p_{14}$ & $\p_{18}$ & $\p_{20}$ & $\p_{24}$ & $\p_{30}$ & $x$ & $y$  \\
\hline
$4$     & $16$   & $24$    & $28 $      & $36$       & $40$       & $48$   & $60$    & $12$ & $20$
\end{tabular}
\end{center}
\item Let $Y \to B$ denote the compactification of $X \to B$ described in \cite[Lemma 4.9]{Tho13}. Let $P : B \to Y$ denote the unique section at infinity (so that $Y = X \cup P$). Then for each $b \in B(k)$ such that $\Delta(b) \neq 0$, the following diagram commutes:
\[ \xymatrix{ X_b(k) \ar[r]^-{\iota_b} \ar[d]_{\eta_b} & G(k) \backslash V_b(k) \ar[d]^{\gamma_b} \\
J_b(k) \ar[r]^-{\del_b} & H^1(k, J_b[2]), } \]
where the maps in the diagram are specified as follows. The top arrow $\iota_b$ is induced by the inclusion $X \hookrightarrow V$. The left arrow $\eta_b$ is the restriction of the Abel--Jacobi map $Q \mapsto [(Q) - (P)]$. To define $\gamma_b$, we use Proposition \ref{prop_parameterization_of_orbits} to obtain an injective homomorphism to $G(k) \backslash V_b(k) \to H^1(k, Z_G(\kappa_b))$, and then compose with the identification $Z_G(\kappa_b) \cong J_b[2]$ of Theorem \ref{thm_identification_of_stabilizer}. The bottom arrow $\del_b$ is the connecting homomorphism associated to the Kummer exact sequence 
\[ \xymatrix@1{ 0 \ar[r] & J_b[2] \ar[r] & J_b \ar[r]^-{\times 2} & J_b \ar[r] & 0.} \]
\end{enumerate}
\end{theorem}

\begin{lemma}\label{lem_image_of_P2}
In Case $\mathbf{E_7}$, suppose $b \in B(k)$ is such that $\Delta(b) \neq 0$. Then $\del_b([(P_2) - (P_1)])$ is in the image of $G(k)\backslash V_b(k)$ under $\gamma_b$, and $\del_b([(P_2) - (P_1)])$ is nontrivial if and only if $H^0(k, Z_G(\kappa_b)) = H^0(k, Z_H(\kappa_b)[2])$.
\end{lemma}
\begin{proof}
Let $\omega \in \Omega$ be the nontrivial element, and let $E' = \sum_{\al \in S_H} e_{\omega(\al)}$. Then $E'$ is a regular nilpotent element of $V$. Since $H^\theta(k)$ acts simply transitively on the set of such elements, there is a unique element $w \in H^\theta(k)$ lifting $\omega$ such that $w(E) = E'$. Let $\kappa'$ denote the Kostant section corresponding to $E'$. Then $w \kappa = \kappa'$ and so $\kappa_b' = w\kappa_b$.
We claim that $\gamma_b(\kappa'_b) = \del_b([(P_2) - (P_1)])$.
The proof is essentially the same as the proof of \cite[Theorem 5.3]{Tho17}, but for the convenience of the reader, we give the details here. Let $\overline{X}_b$ be the base change of $X_b$ to the fixed separable closure $k^s/k$, and define $\overline{Y}_b$ similarly. There is a short exact sequence of \'etale homology groups:
\begin{equation}\label{eqn-SES1}
\xymatrix@1{ 0 \ar[r] & \mu_2 \ar[r] & H_1(\overline{X}_b, \bbF_2) \ar[r] & H_1(\overline{Y}_b, \bbF_2) \ar[r] & 0. } 
\end{equation}
There is a natural symplectic duality on $H_1(\overline{X}_b, \bbF_2)$ which has radical $\mu_2$, and which descends to the usual Poincar\'e duality (or Weil) pairing on $H_1(\overline{Y}_b, \bbF_2) = J_b[2]$. Through an explicit calculation, one can see that $\del_b([(P_2) - (P_1)])$ is the image of the nontrivial element of $\mu_2$ under the connecting homomorphism associated to the dual short exact sequence
\[ \xymatrix@1{ 0 \ar[r] & J_b[2] \ar[r] & H^1(\overline{X}_b, \bbF_2) \ar[r] & \mu_2 \ar[r] & 0, } \]
where we have used the Weil pairing to identify $J_b[2]$ with its dual. 

Let $H^\text{sc}$ denote the simply connected cover of $H$ with centre $A_{H^{\text{sc}}}$. Note that $\theta$ lifts naturally to an automorphism of $H^\text{sc}$, which will again denote by $\theta$, and that because $H^\text{sc}$ is simply connected, the fixed-point subgroup $G' := (H^\text{sc})^\theta$ is connected \cite[Theorem 8.1]{Ste68}. Let $C = Z_H(\kappa_b)$ and let $C^\text{sc} = Z_{H^\text{sc}}(\kappa_b)$. Then $C \subset H$ and $C^\text{sc} \subset H^\text{sc}$ are maximal tori, and we have $Z_{G'}(\kappa_b) = C^\text{sc}[2]$ and $Z_{G}(\kappa_b) = \im(C^\text{sc}[2] \to C[2])$.
It follows from the proof of \cite[Theorem 4.10]{Tho13} that the short exact sequence (\ref{eqn-SES1}) is isomorphic to
\[ \xymatrix@1{ 0 \ar[r] & A_{H^\text{sc}} \ar[r] & C^\text{sc}[2] \ar[r] & Z_G(\kappa_b) \ar[r] & 0,} \]
and its dual is isomorphic to
\begin{equation}\label{eqn-SES4}
\xymatrix@1{ 0 \ar[r] & Z_G(\kappa_b) \ar[r] & C[2] \ar[r] & \pi_0(H^\theta) \ar[r] & 0,} 
\end{equation}
where we have used the $W_H$-invariant duality on $X_\ast(C)$ and the isomorphism $C[2] / Z_G(\kappa_b) \cong \pi_0(H^\theta)$; see also \cite[Corollary 2.12]{Tho13}, which states that this Weyl-invariant duality descends to a non-degenerate symplectic alternating duality on $Z_G(\kappa_b)$. 

Therefore to prove the claim we must show that $\gamma_b(\kappa_b')$ is equal to the image in $H^1(k, Z_G(\kappa_b))$ of the nontrivial element of $\pi_0(H^\theta)$ under the connecting homomorphism associated with the short exact sequence (\ref{eqn-SES4}). This follows from a computation with cocycles. 
Indeed, the second part of Proposition \ref{prop_parameterization_of_orbits} asserts that there exists $g \in G(k^s)$ such that $\kappa'_b = g\kappa_b$. Then the cohomology class $\gamma_b(\kappa_b')$ is represented by the cocycle $\sigma \mapsto g^{-1}( {}^\sigma g)$. But $c := g^{-1}w \in Z_{H^\theta}(\kappa_b) = C[2]$ is a lift of the nontrivial element of $\pi_0(H^\theta)$, so the claim follows from the fact that ${}^\sigma c c^{-1} = ({}^\sigma c c^{-1})^{-1}= g^{-1}({}^\sigma g)$ for all $\sigma \in \Gal(k^s / k)$.

We have established the claim, and the first part of the lemma. To finish the the proof, we note that $\del_b([(P_2) - (P_1)])$ is nontrivial if and only if the connecting homomorphism $\pi_0(H^\theta) \to H^1(k, Z_G(\kappa_b))$ is injective. By exactness, this is equivalent to the surjectivity of the map $H^0(k, Z_G(\kappa_b)) \to H^0(k, C[2])$, which is exactly the criterion given in the statement of the lemma.
\end{proof}
\begin{corollary}\label{cor_generic_non-triviality_of_trivial_divisor_class}
In Case $\mathbf{E_7}$, let $b \in B(k)$ be such that $\Delta(b) \neq 0$, and let $C = Z_H(\kappa_b)$. Suppose that the map $\Gal(k^s / k) \to W(H_{k^s}, C_{k^s})$ induced by the action of $\Gal(k^s/k)$ on $C_{k^s}$ is surjective. 
Then $\del_b([(P_2)-(P_1)])$ is nontrivial in $H^1(k, J_b[2])$.
\end{corollary}
\begin{proof}
By the lemma, it is equivalent to show that the map $H^0(k, Z_G(\kappa_b)) \to H^0(k, C[2])$ is surjective. We have $H^0(k, C[2]) = C^{W(H, C)}[2](k) = Z_H[2](k)$. Since the group $H$ is adjoint, the centre $Z_H$ is trivial, so the map $H^0(k, Z_G(\kappa_b)) \to H^0(k, C[2])$ is clearly surjective.
\end{proof}

\subsection{Reducibility conditions}\label{section-reducible}
We now define the notion of $k$-reducibility and study the properties of $k$-reducible elements of $V(k)$.
\begin{definition}\label{def_reducible}
Let $v \in V$. We say that $v$ is \emph{$k$-reducible} if $\Delta(v) = 0$ or if $v$ is $G(k)$-conjugate to an element of a Kostant section. Otherwise, we say that $v$ is \emph{$k$-irreducible}.
\end{definition} 
The factors of the Cartan decomposition $\frh = \frt \oplus \bigoplus_{\alpha \in \Phi_H} \frh_\alpha$ are invariant under the action of $\theta$; this leads to a corresponding decomposition
\begin{equation}\label{eqn_cartan_decomposition}
V = \bigoplus_{\al \in \Phi_V} \frh_{\al}.
\end{equation}
If $v \in V$, then we write $v = \sum_{\al \in \Phi_V} v_\al$ for the corresponding decomposition of $v$ as a sum of $T$-eigenvectors. Now choose a set of simple roots $S_G = \{\be_1, ..., \be_r\}$ of $\Phi_G$. Since the $\be_i$ form a basis for $X^*(T) \otimes \bbQ$, each element $\gamma \in X^*(T)$ may be written uniquely as $\gamma = \sum_{i =1}^r n_i(\gamma)\be_i$ for some $n_i(\gamma) \in \bbQ$. Our choice of simple roots $S_H \subset \Phi_H$ determines a set of positive roots $\Phi_H^+$. We write $\Phi_V^+$ for $\Phi_H^+ \cap \Phi_V$. 
\begin{lemma}\label{lem-reducible}
Let $v \in V$ and decompose $v$ as $\sum_{\al \in \Phi_V} v_\al$ as in  \textup{(}\ref{eqn_cartan_decomposition}\textup{)}. Suppose one of the following holds:
\begin{enumerate}
\item There exist rational numbers $a_1,  \dots, a_r$ not all equal to zero such that if $\al \in \Phi_V$ and $v_\al \neq 0$, then $\sum a_i n_i(\al) \leq 0$.
\item There exists $w \in \WStab$ such that $v_\al = 0$ if $\al \in w(\Phi_V^+ - S_H)$. 
\end{enumerate}
Then $v$ is $k$-reducible.
\end{lemma}
\textup{(}We recall that the subgroup $\WStab\subset W_H$ was defined in Lemma \ref{lem_weyl_group_represents_component_group}.\textup{)}
\begin{proof}
For the first part of the lemma, we will apply the criterion of Corollary \ref{cor_hilbert_mumford_criterion}. This corollary implies that if $v \in V$ and there exists a nontrivial cocharacter $\lambda \in X_\ast(T)$ such that $v$ has no (strictly) positive weights with respect to $\lambda$, then $\Delta(v) = 0$. Let $\{\check{\omega}_1, ..., \check{\omega}_r\} \subset X_\ast(T) \otimes \bbQ$ be the basis dual to the basis $\{\be_1, \dots, \be_r\}$ of $X^\ast(T) \otimes \bbQ$, and let $\lam =  \sum_{i=1}^r a_i \check{\omega}_i$. Then there exists a positive integer $m$ such that $m\lam \in X_\ast(T)$. The weights of $v$ with respect to $m\lambda$ are exactly the values $\langle \al, m\lambda \rangle = m\sum_{i=1}^r a_i n_i(\al)$ for those $\al \in \Phi_V$ such that $n_i(\al) \neq 0$, so $v$ has no positive weights with respect to $m\lam$. 

For the second item, let $E = \sum_{\alpha \in S_H} e_\alpha$, where each $e_\al$ is a root vector of our fixed Chevalley basis (see Section \ref{sec_vinberg_representations}). Then $E$ is a regular nilpotent element of $V$, and is therefore contained in a unique normal $\frs\frl_2$-triple, which in turn determines a Kostant section $\kappa \subset V$ (see Proposition \ref{prop_parameterization_of_orbits}). Suppose that the vector $v \in V$ satisfies the condition $v_\al = 0$ if $\al \in \Phi_V^+ - S_H$. We may assume that if $\al \in S_H$, then $v_\al \neq 0$; otherwise $v$ also satisfies the condition in the first part of the lemma. In this case, exactly the same argument as in the proof of \cite[Lemma 2.6]{Tho15} shows that $v$ is $G(k)$-conjugate to an element of $\kappa$, hence is $k$-reducible.

Now suppose that there is a nontrivial element $w \in \WStab$ such that the vector $v \in V$ satisfies the condition $v_\al = 0$ if $\al \in w(\Phi_V^+ - S_H)$. We can again assume that $v_\al \neq 0$ if $\al \in w(S_H)$. Let $E' = \sum_{\alpha \in w(S_H)} e_\alpha$, and let $\kappa'$ be the Kostant section corresponding to $E'$. Since the group $H^\theta(k)$ acts simply transitively on the set of regular nilpotents of $V$ (\cite[Lemma 2.14]{Tho13}), there is a unique element $x \in H^\theta(k)$ such that $x \cdot E' = E$. Then $x$ normalizes the torus $T$, since $\frt = \Lie(T)$ is the unique Cartan subalgebra of $\frh$ containing the semisimple parts of the normal $\frs\frl_2$-triples containing $E$ and $E'$ respectively. Thus $x$ corresponds to an element of the Weyl group $W_H$; since $W_H$ acts simply transitively on the set of root bases of $H$, we see that $x$ is a representative in $H^\theta(k)$ of $w$. As in the previous paragraph, the proof of \cite[Lemma 2.6]{Tho15} shows that $x^{-1} v$ is $G(k)$-conjugate to an element of $\kappa$, hence that $v$ is $G(k)$-conjugate to an element of $\kappa'$.
\end{proof}

Given a subset $M \subset \Phi_V$, we define the linear subspace
\[ V(M) = \{ v \in V \mid v_\al = 0 \text{ for all } \al \in M \} \subset V. \]
\begin{proposition}\label{prop_reducibility_conditions}
Let  $M$ be a subset of $\Phi_V$, and suppose that one of the following three conditions is satisfied:
\begin{enumerate}
\item There exists $w \in \WStab$ such that $w(\Phi_V^+ - S_H) \subset M$.
\item There exist integers $a_1, \dots, a_r$ not all equal to zero such that if $\al \in \Phi_V$ and $\sum_{i=1}^r  a_i n_i(\al) > 0$, then $\al \in M$.
\item There exist $\be \in S_G$, $\al \in \Phi_V - M$, and integers $a_1, \dots, a_r$ not all equal to zero such that the following conditions hold:
\begin{enumerate}
\item We have $\{\gamma \pm \be \mid \gamma \in M\} \cap \Phi_V \subset M$.
\item $\al - \be \in \Phi_V - M$.
\item If $\gamma \in \Phi_V$ and $\sum_{i=1}^r a_i n_i(\gamma) > 0$, then $\gamma \in M \cup \{ \al \}$.
\end{enumerate}
\end{enumerate}
Then every element of $V(M)(k)$ is $k$-reducible.
\end{proposition}
\begin{proof}
If either of the first two conditions is satisfied, then the desired reducibility follows from Lemma \ref{lem-reducible}. 
We now show that if the third condition is satisfied, then every element of $V(M)(k)$ is $k$-reducible. Let $v \in V(M)(k)$. If $v_\al = 0$, then $v \in V(M \cup \{ \al \})(k)$, and so $v$ is $k$-reducible by the second part of the proposition. We can therefore assume that $v_\al \neq 0$.

Let $V_{M} = \{ v \in V \mid v_\gamma = 0 \text{ for all } \gamma \in \Phi_V - M \}$. Then there is a $T$-invariant direct sum decomposition $V = V(M) \oplus V_{M}$. Fix a homomorphism $\SL_2 \to G_\be$ where $G_\be$ is the subgroup of $G$ generated by the root groups corresponding to $\be$ and $-\be$.
Condition (a) implies that the decomposition $V = V(M) \oplus V_{M}$ is $G_\be$-invariant. Since the ambient group $H$ is simply laced, the $\be$-root string through $\al$ has length two, and thus $\frh_\al \oplus \frh_{\al-\be}$ is an irreducible $G_\be$-submodule of $V$.
The existence of an irreducible representation of degree two implies that $G_\be \simeq \SL_2$.

Since $\SL_2(k)$ acts transitively on the non-zero vectors in the unique two-dimensional irreducible representation of $\SL_2$, we can find $g \in G_\be(k) \subset G(k)$ such that $(g v)_\al = 0$. This shows that $gv \in V(M \cup \{ \al \})$, hence that $v$ is $k$-reducible, as required.
\end{proof}

\subsection{Roots and weights}

We conclude Section 2 by fixing coordinates in $H$ and $G$. From now on we assume $H$ has type $E_7$ or type $E_8$. As above we let $\Phi_H^+$ be the set of positive roots corresponding to our choice of root basis $S_H$. Similarly, we define $\Phi_H^- \subset \Phi_H$ to be the subset of negative roots.
We note that there exists a unique choice of root basis $S_G$ of $\Phi_G$ such that the positive roots $\Phi_G^+$ determined by $S_G$ are given by $\Phi_G^+ = \Phi_G \cap \Phi_H^+$. 
Indeed, this follows from a consideration of Weyl chambers: the Weyl chambers for $H$ (resp. $G$) are in bijection with the root bases of $\Phi_H$ (resp. $\Phi_G$), and each Weyl chamber for $H$ is contained in a unique Weyl chamber for $G$. If $C_H$ is the fundamental Weyl chamber of $H$ corresponding to $S_H$, and $C_G$ is the unique Weyl chamber for $G$ containing $C_H$, then defining $S_G$ to be the root basis corresponding to $C_G$ yields the desired property. We note that the set of negative roots $\Phi_G^-$ determined by $S_G$ is given by $\Phi_G^- = \Phi_G \cap \Phi_H^-$.

We will later need to carry out explicit calculations, so we now define $S_G$ in terms of the simple roots of $S_H$ in each case $\mathbf{E_7}$ and $\mathbf{E_8}$. We number the simple roots of $H$ and $G$ as in Bourbaki \cite[Planches]{Bou68}.

\subsubsection{Case $\mathbf{E_7}$}

We have $S_H = \{ \alpha_1, \dots, \alpha_7 \}$, where the Dynkin diagram of $H$ is as follows:
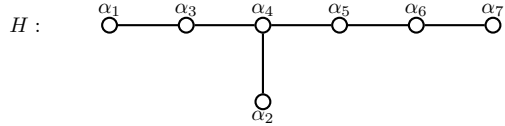
\begin{center}
\begin{tikzpicture}[transform shape, scale=.8]
\node (a) {$H:$}; 
\node[root] (b) [right=of a] {}; 
\node[root] (c) [right=of b] {};
\node[root] (d) [right=of c] {};
\node[root] (e) [right=of d] {};
\node[root] (f) [right=of e] {};
\node[root] (g) [right=of f] {};
\node[root] (h) [below=of d] {};
\node [above] at (b.north) {$\alpha_1$};
\node [above] at (c.north) {$\alpha_3$};
\node [above] at (d.north) {$\alpha_4$};
\node [above] at (e.north) {$\alpha_5$};
\node [above] at (f.north) {$\alpha_6$};
\node [above] at (g.north) {$\alpha_7$};
\node [below] at (h.south) {$\alpha_2$};
\draw[thick] (b) -- (c) -- (d) -- (e) -- (f)--(g);
\draw[thick] (d) -- (h);
\end{tikzpicture} 
\end{center} 
The root basis $S_G = \{ \be_1, \dots, \be_7 \}$ described above consists of the roots
\begin{eqnarray*}
\be_1 &=& \al_3 + \al_4\\
\be_2 &=& \al_5 + \al_6\\
\be_3 &=& \al_2 + \al_4\\
\be_4 &=& \al_1 + \al_3\\
\be_5 &=& \al_4 + \al_5\\
\be_6 &=& \al_6 + \al_7\\
\be_7 &=& \al_2 + \al_3 + \al_4 + \al_5
\end{eqnarray*}
where the Dynkin diagram is as follows:
\begin{center}
\begin{tikzpicture}[transform shape, scale=.8]
\node (a) {$G:$}; 
\node[root] (b) [right=of a] {}; 
\node[root] (c) [right=of b] {};
\node[root] (d) [right=of c] {};
\node[root] (e) [right=of d] {};
\node[root] (f) [right=of e] {};
\node[root] (g) [right=of f] {};
\node[root] (h) [right=of g] {};
\node [above] at (b.north) {$\be_1$};
\node [above] at (c.north) {$\be_2$};
\node [above] at (d.north) {$\be_3$};
\node [above] at (e.north) {$\be_4$};
\node [above] at (f.north) {$\be_5$};
\node [above] at (g.north) {$\be_6$};
\node [above] at (h.north) {$\be_7$};
\draw[thick] (b) -- (c) -- (d) -- (e) -- (f)--(g)--(h);
\end{tikzpicture} 
\end{center} 
We note that the existence of a diagram automorphism for $G$ implies that there are two possible choices of numbering of the roots in $S_G$ consistent with the conventions of Bourbaki; we keep the above choice for the rest of this paper.

\subsubsection{Case $\mathbf{E_8}$}

We have $S_H = \{ \alpha_1, \dots, \alpha_8 \}$, where the Dynkin diagram of $H$ is as follows:
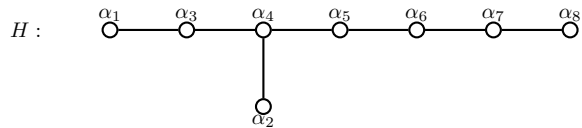
\begin{center}
\begin{tikzpicture}[transform shape, scale=.8]
\node (a) {$H:$}; 
\node[root] (b) [right=of a] {}; 
\node[root] (c) [right=of b] {};
\node[root] (d) [right=of c] {};
\node[root] (e) [right=of d] {};
\node[root] (f) [right=of e] {};
\node[root] (g) [right=of f] {};
\node[root] (h) [right=of g] {};
\node[root] (i) [below=of d] {};
\node [above] at (b.north) {$\alpha_1$};
\node [above] at (c.north) {$\alpha_3$};
\node [above] at (d.north) {$\alpha_4$};
\node [above] at (e.north) {$\alpha_5$};
\node [above] at (f.north) {$\alpha_6$};
\node [above] at (g.north) {$\alpha_7$};
\node [above] at (h.north) {$\alpha_8$};
\node [below] at (i.south) {$\alpha_2$};
\draw[thick] (b) -- (c) -- (d) -- (e) -- (f)--(g)--(h);
\draw[thick] (d) -- (i);
\end{tikzpicture} 
\end{center} 
The root basis $S_G = \{ \be_1, \dots, \be_8 \}$ described above consists of the roots
\begin{eqnarray*}
\be_1 &=& \al_2 + \al_3 + \al_4 + \al_5\\
\be_2 &=& \al_6 + \al_7\\
\be_3 &=& \al_4 + \al_5\\
\be_4 &=& \al_1 + \al_3\\
\be_5 &=& \al_2 + \al_4\\
\be_6 &=& \al_5 + \al_6\\
\be_7 &=& \al_7 + \al_8\\
\be_8 &=& \al_3 + \al_4
\end{eqnarray*}
where the Dynkin diagram is as follows:
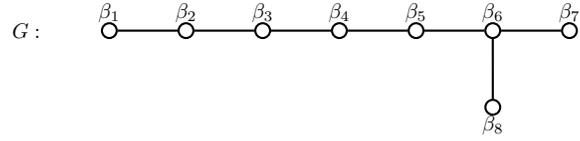
\begin{center}
\begin{tikzpicture}[transform shape, scale=.8]
\node (a) {$G:$}; 
\node[root] (b) [right=of a] {}; 
\node[root] (c) [right=of b] {};
\node[root] (d) [right=of c] {};
\node[root] (e) [right=of d] {};
\node[root] (f) [right=of e] {};
\node[root] (g) [right=of f] {};
\node[root] (h) [right=of g] {};
\node[root] (i) [below=of g] {};
\node [above] at (b.north) {$\be_1$};
\node [above] at (c.north) {$\be_2$};
\node [above] at (d.north) {$\be_3$};
\node [above] at (e.north) {$\be_4$};
\node [above] at (f.north) {$\be_5$};
\node [above] at (g.north) {$\be_6$};
\node [above] at (h.north) {$\be_7$};
\node [below] at (i.south) {$\be_8$};
\draw[thick] (b) -- (c) -- (d) -- (e) -- (f)--(g)--(h);
\draw[thick] (g) -- (i);
\end{tikzpicture} 
\end{center} 
Once again the existence of a diagram automorphism for $G$ means that there are two possible choices of numbering of the roots in $S_G$ consistent with Bourbaki; we  keep the above choice for the rest of this paper.

\section{Integral structures, measures, and orbits}\label{sec_integral_structures}

In Section \ref{sec_a_stable_grading}, we introduced the following data:
\begin{itemize}
\item the group $H$ over $k$, together with split maximal torus $T \subset H$, root basis $S_H \subset X^\ast(T)$, involution $\theta = \Ad \check{\rho}(-1)$, and Lie algebra $\frh = \Lie H$;
\item the group $G = (H^\theta)^\circ$ and its representation on $V = \frh^{d \theta = -1}$, together with a root basis $S_G \subset X^\ast(T)$ and Lie algebra $\frg = \Lie G$;
\item the categorical quotient $B = V \dquot G$ and quotient map $\pi : V \to B$;
\item the discriminant polynomial $\Delta \in k[B]$.
\end{itemize}
From now on, we also fix the regular nilpotent element $E = \sum_{\alpha \in S_H} e_\alpha \in V$. 
We now assume that $k = \bbQ$ and study integral structures on these objects.

\subsection{Integral structures and measures}\label{subsec_measures}

Our choice of Chevalley basis of $\frh$ with root vectors $\{e_\al \mid \al \in \Phi_H\}$ determines a Chevalley basis of $\frg$, with root vectors $\{e_\al \mid \al \in \Phi_G\}$. It hence determines $\bbZ$-forms $\frh_\bbZ \subset \frh$ and $\frg_\bbZ \subset \frg$ (in the sense of \cite{Bor70}). Moreover, $\cV = V \cap \frh_\bbZ$ is an admissible $\bbZ$-lattice that contains $E$.

We extend $G$ to a group scheme over $\bbZ$ given by the Zariski closure of the group $G$ in $\GL(\cV)$. By abuse of notation, we also refer to this $\bbZ$-group scheme as $G$. Then the group $G(\bbZ)$ acts on the lattice $\cV(\bbZ) \subset V(\bbQ)$. The Cartan decomposition $V = \oplus_{\al \in \Phi_V} \frh_\al$ is defined over $\bbZ$, so extends to a decomposition $\cV = \oplus_{\al \in \Phi_V} \cV_\al$. Since there exists a subregular nilpotent element in $V = \cV(\bbQ)$, we may choose a subregular nilpotent element $e \in \cV(\bbZ)$. In Case $\mathbf{E_7}$, we impose the additional condition that $E$ corresponds to $P_1$ in the sense described in Theorem \ref{thm_representation-theoretic_descent_diagram_case_E7}.

Fix a maximal compact subgroup $K \subset G(\bbR)$. Let $P = TN \subset G$ be the Borel subgroup corresponding to the root basis $S_G$, and let $\overline P = T\overline N \subset G$ be the opposite Borel subgroup. Given $c \in \bbR$, we define $T_c = \{ t \in T(\bbR)^\circ \mid \be(t) \leq c \text{ for all } \be \in S_G \}$. 
\begin{proposition}\label{prop_unique_cusp}
We can find a compact subset $\omega \subset \overline{N}(\bbR)$ and a constant $c > 0$ such that $G(\bbA) = G(\bbQ) \cdot (G(\widehat{\bbZ}) \times \Siegel)$, where $\Siegel = \omega T_cK$.
\end{proposition}
\begin{proof}
It suffices to show that $G(\bbA^\infty) = G(\bbQ) \cdot G(\widehat{\bbZ})$ and that we can choose $\Siegel$ so that $G(\bbZ) \cdot \Siegel = G(\bbR)$. This is true: see \cite[\S 6]{Bor66}, \cite[Theorem 4.15]{Pla94}, and \cite[Theorem 8.11, Corollary 2]{Pla94}.
\end{proof}
Henceforth we fix a choice of $\Siegel = \omega T_cK$ as in Proposition \ref{prop_unique_cusp}.

After rescaling the polynomials $\p_i \in \bbQ[V]^G$  and $x, y \in \bbQ[X]$ appearing in Theorem \ref{thm_representation-theoretic_descent_diagram_case_E7} (resp. Theorem \ref{thm_representation-theoretic_descent_diagram_case_E8}), we can assume that each polynomial $\p_i$ lies in $\bbZ[\cV]^G$. We define $\cB = \Spec \bbZ[\p_2, \p_6, \dots, \p_{18} ]$ in Case $\mathbf{E_7}$ (resp. $\Spec \bbZ[\p_2, \p_8, \dots, \p_{30}]$ in Case $\mathbf{E_8}$), and write $\pi : \cV \to \cB$ for the natural morphism, which recovers our existing map $\pi : V \to B$ after extension of scalars to $\bbQ$. If $b \in \cB(\bbR) = B(\bbR)$, then we define the height of $b$ to be
\begin{equation}
\Ht(b) = \sup_i | \p_i(b) |^{\deg \Delta / i}.
\end{equation}
If $v \in V(\bbR)$, then we define $\Ht(v) = \Ht(\pi(v))$. Since $\deg \p_i = i$, the height function is homogeneous: for all $\lambda \in \bbR^\times$, we have $\Ht(\lambda v) = | \lambda |^{\deg \Delta} \Ht(v)$.

We define $\cX = \Spec \bbZ[x, y, \p_2, \p_6, \dots, \p_{18}]$ in Case $\mathbf{E_7}$ (resp. $\Spec \bbZ[x, y, \p_2, \p_8, \dots, \p_{30}]$ in Case $\mathbf{E_8}$). Thus $\cX$ is isomorphic to affine space $\bbA_\bbZ^{r+2}$, and the morphism $X \to B$ naturally extends to a morphism $\cX \to \cB$, still given in coordinates by the equation (\ref{eqn_intro_curves_of_type_E7}) in Case $\mathbf{E_7}$ (resp. (\ref{eqn_intro_curves_of_type_E8}) in Case $\mathbf{E_8}$). For any ring $R$ and any subset $A \subset \cV(R)$, we write $A^\text{reg.ss.}$ for $\{a \in A \mid \Delta(a) \neq 0\}$. Similarly if $A' \subset \cB(R)$ then we write $(A')^\text{reg.ss.}$ for the set $\{ a \in A' \mid \Delta(a) \neq 0 \}$.

Fix a left-invariant top form $\omega_G$ on $G$; it is determined uniquely up to multiplication by $\bbZ^\times = \{ \pm 1 \}$.  For any place $v$ of $\bbQ$, we define a Haar integral on $G(\bbQ_v)$ using the volume element $dg = \lvert \omega_G \rvert_v$. 

If $\bbQ_v = \bbR$, then we can use the Iwasawa decomposition on $G(\bbR) = T(\bbR)^\circ \overline N(\bbR)K = \overline N(\bbR)T(\bbR)^\circ K$ to decompose $dg = dt~ dn ~dk$ on $G(\bbR)$ as follows (cf. \cite[Section 2.7]{Tho15}). We give $T(\bbR)^\circ$ the measure pulled back from the isomorphism $\prod_{\al \in S_G} \al : T(\bbR)^\circ \simeq \bbR_{>0}^r$. We give $K$ its normalized (probability) Haar measure. We then choose the unique Haar measure $dn$ on $\overline N(\bbR)$ such that $dg = dt~dn~dk$. For $t \in T(\bbR)$, we define $\del_G(t) = \prod_{\al \in \Phi_G^-} \al(t)$. Then for any continuous compactly supported function $f : G(\bbR) \to \bbC$, we have the equalities
\[ \int_{g \in G(\bbR)} f(g) \, dg = \int_{t \in T(\bbR)^\circ} \int_{n \in \overline{N}(\bbR)} \int_{k \in K} f(tnk) \, dk \, dn \, dt = \int_{t \in T(\bbR)^\circ} \int_{n \in \overline{N}(\bbR)} \int_{k \in K} f(ntk) \delta_G(t)^{-1} \, dk \, dn \, dt. \]
We also define measures on $V$ and $B$ as in \cite[Section 2.8]{Tho15} by fixing an invariant differential top form $\omega_V$ on $\cV$ and by defining $\omega_B = d\p_{2} \wedge d\p_{6} \wedge ... \wedge d\p_{18}$ in Case $\mathbf{E_7}$ (resp. $\omega_B = d\p_2 \wedge d\p_8 \wedge ... \wedge d\p_{30}$ in Case $\mathbf{E_8}$). If $v$ is a place of $\bbQ$, then the formulae $db = | \omega_B |_v$ and $dv = | \omega_V |_v$ define measures on $B(\bbQ_v)$ and $V(\bbQ_v)$ respectively. Fixing these choices, we have the following useful result.
\begin{lemma}\label{lem_algebraic_decomposition_of_measure}
There exists a rational number $W_0 \in \bbQ^\times$ with the following property: let $k' / \bbQ$ be any field extension, and let $\frc \subset V(k')$ be a Cartan subspace. Let $\mu_\frc : G_{k'} \times \frc \to V_{k'}$ be the natural action map. Then $\mu_\frc^\ast \omega_V = W_0 \omega_G \wedge \pi|_\frc^\ast \omega_B$.
\end{lemma}
\begin{proof}
The proof is identical to that of  \cite[Proposition 2.13]{Tho15}.
\end{proof}
\begin{proposition}\label{prop_description_of_measures_padic_case}
Let $p$ be a prime. 
\begin{enumerate} \item Let $\phi : \cV(\bbZ_p)^\text{reg.ss.} \to \bbR$ be a function of compact support that is locally constant (resp. continuous) and invariant under the action of $G(\bbZ_p)$. Then the function $F_\phi : B(\bbQ_p)^\text{reg.ss.} \to \bbR$ defined by the formula
\[ F_\phi(b) = \sum_{v \in G(\bbZ_p) \backslash \cV_b(\bbZ_p)} \frac{\phi(v)}{\# \Stab_{G(\bbZ_p)}(v)} \]
is of compact support and locally constant (resp. continuous), and we have the formula
\[ \int_{v \in \cV(\bbZ_p)} \phi(v) \, dv = |W_0|_p \vol(G(\bbZ_p)) \int_{b \in \cB(\bbZ_p)} F_\phi(b) \, db. \]
\item Define a function $m_p : \cV(\bbZ_p)^\text{reg.ss.} \to \bbR$ by the formula
\[ m_p(v) = \sum_{v' \in G(\bbZ_p) \backslash (G(\bbQ_p) \cdot v \cap \cV(\bbZ_p))} \frac{ \# \Stab_{G(\bbQ_p)}(v) }{ \# \Stab_{G(\bbZ_p)}(v')}. \]
Then $m_p$ is locally constant.
\item  Let $\psi : \cV(\bbZ_p)^\text{reg.ss.} \to \bbR$ be a continuous function of compact support that is $G(\bbQ_p)$-invariant, in the sense that if $v, v' \in \cV(\bbZ_p)$, $g \in G(\bbQ_p)$, and $gv = v'$, then $\psi(v) = \psi(v')$. Then we have the formula
\[ \int_{v \in \cV(\bbZ_p)} \psi(v) \, dv = | W_0 |_p \vol(G(\bbZ_p)) \int_{b \in \cB(\bbZ_p)} \sum_{v \in G(\bbQ_p) \backslash \cV_b(\bbZ_p)} \frac{m_p(v)\psi(v)}{\# \Stab_{G(\bbQ_p)}(v)} \, db. \]
\end{enumerate}
\end{proposition}

\begin{proof}
The first part follows from Lemma \ref{lem_algebraic_decomposition_of_measure} and the $p$-adic formula for integration in fibres; see \cite[\S 7.6]{Igu00}.
To prove the second part, we note that the function $v \mapsto \# \Stab_{G(\bbQ_p)}(v)$ is locally constant, because the universal stabilizer $Z \to V^\text{reg.ss.}$ is finite \'etale.  It therefore suffices to show that the function
\[ n_p(v) := \sum_{v' \in G(\bbZ_p) \backslash (G(\bbQ_p) \cdot v \cap \cV(\bbZ_p))} \frac{1 }{ \# \Stab_{G(\bbZ_p)}(v')} \]
is locally constant. Suppose $v \in V(\bbQ_p)^\text{reg.ss.}$. Let $\frc \subset V({\bbQ_p})$ be the unique Cartan subspace containing $v$. Since $\pi|_\frc$ is \'etale above $B(\bbQ_p)^\text{reg.ss.}$, we can find an open compact neighbourhood $B_v$ of $\pi(v)$ in $B(\bbQ_p)^\text{reg.ss.}$ such that $\pi^{-1}(B_v) \cap \frc = \sqcup_{i=1}^s U_i$ is a disjoint union of open subsets of $\frc$ and each $\pi|_{U_i} : U_i \to B_v$ is a homeomorphism. Let $U = U_j$ be the open subset containing $v$. Let $\mu : G(\bbQ_p) \times U \to V(\bbQ_p) \cap \pi^{-1}(B_v)$ be the restriction of the natural action map. Then $\mu$ is proper, and so $\mu^{-1}(\cV(\bbZ_p) \cap \pi^{-1}(B_v))$ is compact. It follows that the characteristic function $\chi$ of the set $\mu (\mu^{-1}(\cV(\bbZ_p) \cap \pi^{-1}(B_v)) )\subset \cV(\bbZ_p)^\text{reg. ss.}$ is locally constant and of compact support. For $v' \in U$, we have $n_p(v') = F_\chi(\pi(v'))$, where $F_\chi$ is as defined in the statement of the first part of the proposition. Thus by the first part of the proposition $n_p$ is locally constant. The third part of the proposition follows from the first two.
\end{proof}

\subsection{Selmer elements and integral orbits}

We now discuss the construction of elements of $\cV(\bbZ_p)$ and $\cV(\bbZ)$ from rational points of algebraic curves. 
We first show that certain geometric orbits have integral representatives.
\begin{lemma}\label{lem_integrality_of_Kostant_section}
There exists an integer $N_0 \geq 1$ with the following properties:
\begin{enumerate} \item For any prime $p$ and any $b \in \cB(\bbZ_p)$, we have $N_0\cdot \kappa_b \in \cV(\bbZ_p)$.
\item In Case $\mathbf{E_7}$, let $w \in \Omega$ be the non-trivial element and let $\kappa'$ denote the Kostant section corresponding to the regular nilpotent element $E' = \sum_{\alpha \in S_H} e_{w \al}$. Then for any prime $p$ and for any $b \in \cB(\bbZ_p)$, we have $N_0 \cdot \kappa'_b \in \cV(\bbZ_p)$. 
\item For any prime $p$ and any $x \in \cX(\bbZ_p)$, we have $N_0\cdot x \in \cV(\bbZ_p)$.
\item If $b \in N_0^2\cdot \cB(\bbZ)$, then $b \in \pi(\cV(\bbZ))$.
\end{enumerate}
In the first three items $N_0$ is acting via the $\bbG_m$-action discussed in Section \ref{section_transverse_slices}. In the third item $N_0$ is acting via the natural $\bbG_m$-action on $B$.
\end{lemma}
\begin{proof}
This follows from the existence of the contracting $\bbG_m$-actions on $\kappa$, $\kappa'$, and $\cX$, cf. \cite[Lemma 2.8]{Tho15}.
\end{proof}
\begin{lemma}\label{lem_existence_of_local_integral_orbits}
There exists an integer $N_1 \geq 1$ with the following property: for any prime $p$ and any $b \in N_1 \cdot  \cB(\bbZ_p)$ such that $\Delta(b) \neq 0$, the canonical image of $Y_b(\bbQ_p)$ in $H^1(\bbQ_p, J_b[2])$ is contained in the image of the composite map:
\[ \cV_b(\bbZ_p) \to G(\bbQ_p) \backslash V_b(\bbQ_p) \overset{\gamma_b}\longrightarrow H^1(\bbQ_p, J_b[2]) \]
(where $\gamma_b$ is as in Theorems \ref{thm_representation-theoretic_descent_diagram_case_E7} and \ref{thm_representation-theoretic_descent_diagram_case_E8} for the case when $k = \bbQ_p$).
\end{lemma}
\begin{proof}
We just treat the case when $H$ is of type ${E_7}$; the ${E_8}$ case is more straightforward, since there is only one point at infinity. We will show that we can take $N_1 = 2^4 N_0^2$, where $N_0$ is as in Lemma \ref{lem_integrality_of_Kostant_section}. We recall that the curve $Y_b$ is given by the equation 
\[ y_0^3 z_0 = x_0^3 y_0 + \p_{10} x_0^2 z_0^2 + x_0(\p_2 y_0^2 z_0 + \p_8 y_0 z_0^2 + \p_{14} z_0^3) + \p_6 y_0^2 z_0^2 + \p_{12} y_0 z_0^3 + \p_{18}z_0^4, \]
and has two sections $P_1 = [0 : 1 : 0]$ and $P_2 = [1 : 0 : 0]$ at infinity; the map $Y_b(\bbQ_p) \to J_b(\bbQ_p) / 2 J_b(\bbQ_p)$ sends a point $P$ to the class of the divisor $(P) - (P_1)$. We define $\cY$ to be the closed subscheme of $\bbP^2_\cB$ defined by the same equation; then the complement in $\cY$ of its sections at infinity is naturally identified with $\cX$ by Theorem \ref{thm_representation-theoretic_descent_diagram_case_E7}. For $b \in \cB(\bbQ_p)$, $\cY_b$ is smooth in an open neighbourhood of these sections at infinity. If $t \in \bbQ_p^\times$, then the isomorphism $X_b \to X_{t^2 b}$ induced by the action of $\bbG_m$ on $X$ extends to an isomorphism $Y_b \to Y_{t^2 b}$ that maps $[x_0 : y_0 : z_0]$ to $[ t^8 x_0 : t^{12} y_0 : z_0 ]$.

We first claim that if $b \in 2^4 \cB(\bbZ_p)$, then every divisor class in the image of the map $Y_b(\bbQ_p) \to J_b(\bbQ_p) / 2 J_b(\bbQ_p)$ is represented by either the zero divisor, the divisor $P_2 - P_1$, or a divisor of the form $P - P_1$ for some $P \in \cX_b(\bbZ_p)$.

If $P \in Y_b(\bbQ_p)$, then we write $\overline P$ for the image of $P$ in $\cY_b(\bbF_p)$. The special fibre $\cY_{b, \bbF_p}$ is reduced, and has at most two irreducible components, which are geometrically irreducible. Moreover, if there are two irreducible components, then $\overline{P}_1$ and $\overline{P}_2$ lie on distinct irreducible components. Indeed, due to the presence of the contracting $\bbG_m$-action, any property of the morphism $\cY \to \cB$ which is open on the base can be checked in the central fibre. Thus \cite[Tag 0C0E]{stacks-project} implies that all of the fibres of $\cY$ are geometrically reduced; and then \cite[Tag 055R]{stacks-project} implies that the two sections $P_1, P_2$ together meet all irreducible components in every geometric fibre. In particular, every irreducible component of $\cY_{b, \bbF_p}$ is geometrically irreducible.

Let $\cJ_b = \Pic^0_{\cY_b / \bbZ_p}$ be the open subscheme of $\Pic_{\cY_b / \bbZ_p}$ corresponding to those invertible sheaves that are fibrewise of degree 0 on each irreducible component (see \cite[\S 8.4]{Bos90}). Then $\cJ_b$ is a smooth and separated scheme over $\bbZ_p$ (see \cite[\S 9.4, Theorem 2]{Bos90}).
We note that if $Q \in \cJ_b(\bbZ_p)$ has trivial image in $\cJ_b(\bbZ_p / 2^3 p \bbZ_p)$, then $Q$ is divisible by 2 in $\cJ_b(\bbZ_p)$ (this follows from \cite[Theorem 6.1]{Sil90} and its generalization \cite[Proposition 3.1]{Cla08}).

Let $P = (x, y) \in Y_b(\bbQ_p)$. To prove the claim, it suffices to show that if $P \not\in \cX_b(\bbZ_p)$, then one of the divisor classes $[(P) - (P_1)]$ or $[(P) - (P_2)]$ is divisible by 2 in $J_b(\bbQ_p)$. We can assume that $xy \neq 0$. We note that if $P \not\in \cX_b(\bbZ_p)$, then (at least) one of $x, y$ must be non-integral. If $x$ is integral then the defining equation of $Y_b$ shows that $y$ is integral too. We can therefore write $x = p^m u$, $y = p^n v$, with $u, v \in \bbZ_p^\times$ and $m < 0$. We note that if $n < 0$, then we must have $2n = 3m$, hence we can write $n = 3k$, $m = 2k$ for some $k < 0$.

We first treat the case where $p$ is odd. If $n < 0$, then we have
\[ P = [ p^{2k} u : p^{3k} v : 1] = [p^{-k} u : v : p^{-3k} ] \equiv P_1 \text{ mod }p, \]
and we see that $[(P) - (P_1)]$ is divisible by 2 in $J_b(\bbQ_p)$. If $n \geq 0$, then $P \equiv P_2 \text{ mod }p$, and $[(P) - (P_2)]$ is divisible by 2 in $J_b(\bbQ_p)$. This establishes the claim in the case when $p$ is odd.

Now suppose that $p = 2$. Our assumption $b \in 2^4 \cB(\bbZ_2)$ means that $c_i(b)$ is divisible by $2^{4i}$ for each $i \in \{2, \dots, 18\}$. We write $\iota : Y_b \to Y_{\frac{1}{4} b}$ for the map $[x_0 : y_0 : z_0] \mapsto [2^{-8} x_0 : 2^{-12} y_0 : z_0] = [2^4 x_0 : y_0 : 2^{12} z_0]$. If $n < 0$, then we get
\[ \iota(P) = [2^{4-k} u : v : 2^{12 - 3k} ] \equiv P_1 \text{ mod }2^{4}. \]
This shows that $[(\iota(P)) - (P_1)]$ is divisible by 2 in $J_{\frac{1}{4} b}(\bbQ_2)$, hence $[(P) - (P_1)]$ is divisible by 2 in $J_b(\bbQ_2)$. If $n \geq 0$, then we have $P = [ 1 : 2^{n-m} v / u : 2^{-m} / u ] = [ 1 : w : z ]$, say, and we have an equation
\[ w(1 - w^2 z) = O(2^8 z) \]
in $\bbZ_2$. It follows that $n - m > 8$. Then we get
\[ \iota(P) = [ 2^4 : 2^{n-m} v / u : 2^{12-m} / u ] = [ 1 : 2^{n - m - 4} v / u : 2^{8-m} / u] \equiv P_2 \text{ mod }2^4, \]
hence $[(P) - (P_2)]$ is divisible by 2 in $J_b(\bbQ_2)$. This completes the proof of the claim.

We now show how the claim implies the lemma. We drop our assumption on the parity of $p$, and take $b = N_0^2 c$, where $c \in 2^4 \cB(\bbZ_p)$. Given a class $\phi$ in $H^1(\bbQ_p, J_c[2])$, if $\phi$ is in the image of $Y_c(\bbQ_p)$, then $\phi$ is represented by either $P_1$, $P_2$, or an element of $\cX_c(\bbZ_p)$. Let $\phi'$ denote the corresponding class in $H^1(\bbQ_p, J_b[2])$. If $P_1$ is a representative, then $\kappa_b \in V_b(\bbQ_p)$ represents the corresponding rational orbit. By Lemma \ref{lem_integrality_of_Kostant_section}, we have $\kappa_b = N_0 \cdot \kappa_c \in \cV(\bbZ_p)$, so $\kappa_b$ is even an integral representative for this rational orbit. If $P_2$ is a representative, then $\kappa'_b \in \cV(\bbZ_p)$ is an integral representative, by the same argument. 

Suppose instead that $\phi$ is represented by a divisor $(P) - (P_1)$, where $P \in \cX_c(\bbZ_p)$. Then $\phi'$ is represented by the divisor $(N_0 \cdot P) - (P_1)$, where now $N_0 \cdot P \in N_0 \cdot \cX(\bbZ_p)$. By Lemma \ref{lem_integrality_of_Kostant_section}, we have $N_0 \cdot \cX(\bbZ_p) \subset \cV(\bbZ_p)$, showing that $N_0 \cdot P \in \cV_b(\bbZ_p)$ is an integral representative for the rational orbit corresponding to the class $\phi$. This completes the proof. 
\end{proof}
\begin{proposition}\label{prop_existence_of_global_integral_orbits}
Let $N_1 \in \bbZ_{\geq 1}$ be an as in Lemma \ref{lem_existence_of_local_integral_orbits}. Then for any $b \in N_1  \cdot\cB(\bbZ)$ such that $\Delta(b) \neq 0$, the 2-Selmer set $\Sel_2(Y_b) \subset H^1(\bbQ_p, J_b[2])$ is contained in the image of the composite map 
\[ \cV_b(\bbZ) \to G(\bbQ) \backslash V_b(\bbQ) \overset{\gamma_b}\longrightarrow H^1(\bbQ, J_b[2]). \]
Consequently, for any $b \in \cB(\bbZ)$ such that $\Delta(b) \neq 0$, we have $\# \Sel_2(Y_b) \leq \# G(\bbQ) \backslash \cV_{N_1 \cdot b}(\bbZ)$.
\end{proposition}
\begin{proof} Suppose $c \in \Sel_2(Y_b)$. We first show that $c \in \gamma_b(G(\bbQ)\backslash V(\bbQ))$; by Proposition \ref{prop_parameterization_of_orbits} this is the case exactly when the image $c'$ of $c$ under the map 
\[ H^1(\bbQ, J_b[2]) \to H^1(\bbQ, G) \]
is trivial. By commutativity of the diagram in Theorem \ref{thm_representation-theoretic_descent_diagram_case_E7} in Case $\mathbf{E_7}$ (resp. Theorem \ref{thm_representation-theoretic_descent_diagram_case_E8} in Case $\mathbf{E_8}$) and the definition of the 2-Selmer set, we see that $c'$ is locally trivial, in the sense that its image in $H^1(\bbQ_v, G)$ is trivial for every place $v$ of $\bbQ$. We claim that this implies that $c'$ is itself trivial. Indeed, write $G^{\text{sc}}$ for the simply connected cover of $G$. The centre of $G$ has order 2 in both cases (see, e.g., \cite[Proof of Proposition A.1]{Tho16}). Thus we see that there is a short exact sequence of groups over $\bbQ$:
\[ \xymatrix@1{ 1 \ar[r] & \mu \ar[r] & G^{\text{sc}} \ar[r] & G \ar[r] & 1, } \]
where $\mu = \mu_4$ (in Case $\mathbf{E}_7$) or $\mu_2$ (in Case $\mathbf{E}_8$). This leads to the following commutative diagram of pointed Galois cohomology sets, in which the rows are exact:
\[ \xymatrix{ H^1(\bbQ, \mu) \ar[d]^{\loc_1} \ar[r] & H^1(\bbQ, G^{\text{sc}}) \ar[r] \ar[d]^{\loc_2} & H^1(\bbQ, G) \ar[r] \ar[d]^{\loc_3} & H^2(\bbQ, \mu) \ar[d]^{\loc_4} \\
\prod_v H^1(\bbQ_v, \mu) \ar[r] & \prod_v  H^1(\bbQ_v, G^{\text{sc}}) \ar[r]  & \prod_v H^1(\bbQ_v, G) \ar[r] &\prod_v  H^2(\bbQ_v, \mu). } \]
Since $G^{\text{sc}}$ is simply connected, the map $\loc_2$ is bijective, and  $H^1(\bbQ_p, G^{\text{sc}})$ is trivial for every prime $p$. By class field theory, the map $\loc_4$ is injective. 
Using these facts, a diagram chase shows that the triviality of $\loc_3(c')$ forces $c'$ itself to be trivial. 

We can therefore choose a vector $v \in V_b(\bbQ)$ representing our class $c$. By Lemma \ref{lem_existence_of_local_integral_orbits}, for each prime $p$ there exists an element $g_p \in G(\bbQ_p)$ such that $g_p \cdot v \in \cV_b(\bbZ_p)$. By Proposition \ref{prop_unique_cusp}, there is an element $g \in G(\bbQ)$ such that $g_p \in G(\bbZ_p) g$ for every prime $p$. It follows that $g\cdot v \in \cV_b(\bbZ)$, as required.
\end{proof}

\subsection{Subsets of $V(\bbR)$ and $V(\bbQ_p)$}

We conclude this section by constructing some useful subsets of $V(\bbR)$ and $V(\bbQ_p)$. We first consider $V(\bbR)$. Let $\frc_1, \dots, \frc_n$ denote representatives of the distinct $G(\bbR)$-conjugacy classes of Cartan subspaces of $V(\bbR)$. For each $i \in \{ 1, \dots, n\}$, let $\frc_i'$ denote the closed subset of $\frc_i^\text{reg.ss.}$ given by $\frc_i' = \{ v \in \frc_i^{\text{reg.ss.}} \mid\Ht(v) = 1\}$. Arguing as in \cite[\S 2.9]{Tho15}, we can find a cover of $\frc_i'$ by finitely many connected semialgebraic open subsets $U_{ij}$ such that each map $\pi|_{U_{ij}} : U_{ij} \to \{ b \in B(\bbR)^\text{reg.ss.} \mid \Ht(b) = 1 \}$ is a homeomorphism onto its image. We write $L_1, \dots, L_m$ for the sets $\pi(U_{ij})$ for all $i, j$ in any order, and for $L_k = \pi(U_{i, j})$ we set $s_k := (\pi|_{U_{i,j}})^{-1} : L_k \to V(\bbR)^\text{reg.ss.}$. We can extend $s_k$ to a map $s_k : \Lambda \cdot L_k \to V(\bbR)^\text{reg.ss.}$ by the formula $s_k (\lambda b) = \lambda s_k(b)$ for any $\lambda \in \Lambda$, $b \in L_k$. 
\begin{lemma}\label{lem_construction_of_real_sections}
With notation as above, each map $s_k : \Lambda \cdot L_k \to V(\bbR)^\text{reg.ss.}$ is a semialgebraic map, and $s_k(L_k)$ has compact closure in $V(\bbR)$. The quantity $r_k := \# \Stab_{G(\bbR)}(s_k(b))$ is independent of the choice of $b \in L_k$. We have $\cup_{k=1}^m G(\bbR) \cdot \Lambda \cdot s_k(L_k) = V(\bbR)^\text{reg.ss.}$. For any continuous function $f : V(\bbR)^\text{reg.ss} \to \bbR$ of compact support, we have
\[ \int_{v \in G(\bbR) \cdot \Lambda \cdot L_k} f(v) \, dv = \frac{| W_0 |_\infty}{r_k} \int_{b \in \Lambda \cdot L_k} \int_{g \in G(\bbR)} f(g\cdot s_k(b)) \, dg \, db. \]
Consequently for any $x \geq 1$ we have:
\[ \vol( \Siegel \cdot [1, x^{1/\deg \Delta} ] \cdot s_k(L_k) ) \leq | W_0 |_\infty \vol(\Siegel) \vol( [1, x^{1/\deg \Delta}] \cdot L_k ). \]
\end{lemma}
\begin{proof}
Let $\mu_k : G(\bbR) \times (\Lambda \cdot L_k) \to V(\bbR)^\text{reg.ss.}$ be given by $(g, b) \mapsto g \cdot s_k(b)$. Then $\mu_k$ is a local diffeomorphism onto its image, with fibres of cardinality $r_k$.  By Lemma \ref{lem_algebraic_decomposition_of_measure} we have $\mu_k^\ast \omega_V = W_0 \omega_G \wedge \omega_B$. The displayed formulae follow from this identity.
\end{proof}
We now consider $V(\bbQ_p)$.
\begin{lemma}\label{lem_large_neron_component_group}
There exists a constant $\varep \in (0, 1)$ with the following property: let $p$ be a prime congruent to $1 \mod 6$. Then there exists a non-empty open compact subset $U_p \subset \cB(\bbZ_p)$ such that for all $b \in U_p$, we have $\Delta(b) \neq 0$, $\cX_b(\bbZ_p) \neq \emptyset$, and 
\[ \frac{\#( \im(Y_b(\bbQ_p) \to J_b(\bbQ_p) / 2 J_b(\bbQ_p)) )}{\# J_b(\bbQ_p) /2 J_b(\bbQ_p)} \leq \varep. \]
\end{lemma}
\begin{proof}
Let $p$ be a prime with $p \equiv 1 \text{ mod } 6$. It suffices to show that we can find a single $b \in \cB(\bbZ_p)$ with $\Delta(b) \neq 0$, $\cX_b(\bbZ_p) \neq \emptyset$, and 
\[ \frac{\#( \im(Y_b(\bbQ_p) \to J_b(\bbQ_p) / 2 J_b(\bbQ_p)) )}{\# J_b(\bbQ_p) /2 J_b(\bbQ_p)} < 1. \]
By continuity considerations of the type in \cite[\S 8]{Poo12}, we can then take $U_p$ to be any sufficiently small open compact neighbourhood of $b$ in $\cB(\bbZ_p)$. We will in fact exhibit $b \in U_p$ such that $\Delta(b) \neq 0$, $\cX_b(\bbZ_p) \neq \emptyset$, the component group $\Phi$ of the N\'eron model of $J_b$ is isomorphic to $(\bbZ/2\bbZ)^2$, and the image of $Y_b(\bbQ_p)$ in $\Phi$ is the identity. This will imply that the lemma holds with $\varep = \frac{1}{4}$.

We first return to the $E_6$ family of curves (\ref{eqn_intro_curves_of_type_E6}):
\[ y^3 = x^4 + y(\p_2 x^2 + \p_5 x + \p_8) + \p_6 x^2 + \p_9 x + \p_{12} \]
described in the introduction to this paper. In this case the existence of such a point $b$ is asserted in \cite[Proposition 2.15]{Tho15}. The proof given there is incorrect; more precisely, the description of the special fibre of a regular model of the curve $y^3 = x^4 - p^2$ is incorrect. We will first remedy this error. The calculation in this case will also play a role in the proof of the lemma in Cases $\mathbf{E_7}$ and $\mathbf{E_8}$.

We consider instead the curve given by the equation $y^3 = (x-1)(x^3 - p^2)$. (This curve can be put into the canonical form (\ref{eqn_intro_curves_of_type_E6}) by a linear change of variable in $x$.) Let $\cY$ be the curve inside $\bbP^2_{\bbZ_p}$ given by the projective closure of this equation, and let $\cZ \subset \bbA^2_{\bbZ_p}$ denote the complement of the unique point at infinity. It is clear that $\cZ(\bbZ_p) \neq \emptyset$. Moreover, $\cY$ has a unique point that is not regular, namely the point corresponding to $(x, y) = (0, 0)$ in the special fibre $\cZ_{\bbF_p}$. 

This singularity can be resolved by blowing up. Let $\cY' \to \cY$ denote the blow-up at the unique non-regular point of $\cY$. Then $\cY'$ has exactly 3 non-regular points. The special fibre of $\cY'$ has two irreducible components, namely the strict transform of $\cY_{\bbF_p}$ and a smooth exceptional divisor. Let $\cY'' \to \cY'$ denote the blow-up of the 3 non-regular points. Then $\cY''$ is regular, and the special fibre $\cY''_{\bbF_p}$ has 5 irreducible components: the strict transform $C_1$ of $\cY_{\bbF_p}$, the strict transform $C_5$ of the exceptional divisor in $\cY'_{\bbF_p}$, and the smooth exceptional divisors $C_2, C_3, C_4$ of the blow-up $\cY'' \to \cY$.

We note that blow-up commutes with flat base change, so to verify our claims about the component group $\Phi$ it suffices to perform these blow-ups in the completed local ring of $\cY$ at the maximal ideal $(p, x, y)$, which is in turn isomorphic to $\bbZ_p \llbracket x, w \rrbracket / (w^3 - x^3 + p^2)$. Here we find that all the irreducible components in the special fibre of $\cY''_{\bbF_p}$ are smooth and geometrically irreducible, and their intersection graph is given as follows:
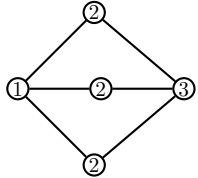
\begin{center}
\begin{tikzpicture}[transform shape, scale=.8]
\node[fatroot] (a) {1}; 
\node[fatroot] (b) [above right=of a] {2}; 
\node[fatroot] (c) [right = of a] {2};
\node[fatroot] (d) [below right = of a] {2};
\node[fatroot] (e) [right = of c] {3};
\draw[thick] (a) -- (b);
\draw[thick] (a) -- (c);
\draw[thick] (a) -- (d);
\draw[thick] (b) -- (e);
\draw[thick] (c) -- (e);
\draw[thick] (d) -- (e);
\end{tikzpicture} 
\end{center} 
All intersections are transverse, and the multiplicities of $C_1, C_2, C_3, C_4$ and $C_5$ are respectively $1, 2, 2, 2$, and $3$. The intersection matrix of the special fibre of $\cY''$ is therefore
\[ M =  \left(\begin{array}{ccccc} -6 & 1 & 1 & 1 & 0 \\ 1 & -2 & 0 & 0 & 1 \\ 1 & 0 & -2 & 0 & 1 \\ 1 & 0 & 0 & -2 & 1 \\ 0 & 1 & 1 & 1 & -2 \end{array}\right). \]
Let $v = (1, 2, 2, 2, 3)$. Then $M v = 0$ and there is an isomorphism $\Phi \cong v^\perp / \im M$, where we consider $v$ as an element of $\bbZ^5$ and $M$ as a $\bbZ$-module homomorphism (see \cite[\S 9.6]{Bos90}). A calculation shows that $\Phi \cong (\bbZ/2\bbZ)^2$, as claimed. Each point of $\cY(\bbZ_p) = \cY''(\bbZ_p)$ reduces modulo $p$ to a smooth point of the special fibre $\cY''_{\bbF_p}$. Since there is exactly one component of $\cY''_{\bbF_p}$ of multiplicity one, we see that all points of $\cY(\bbZ_p)$ reduce to this component; consequently, their image in the N\'eron component group $\Phi$ is trivial (to see this, use the recipe in \cite[\S 5]{Lor00}).

We now turn to Case $\mathbf{E_7}$. Consider a perturbation 
\[ y^3 = (x-1)(x^3 - p^2) + \lambda x^3 y, \]
where $\lambda \in \bbZ_p - \{ 0 \}$. Using the procedure of Proposition \ref{prop_universality_of_E7_family}, we can make a change of variable to put this curve in the form (\ref{eqn_intro_curves_of_type_E7}): the perturbation causes the point $[0 : 1 : 0 ]$ at infinity to be a flex point, but no longer a hyperflex point. One may check that the curve obtained in this way has nontrivial integral points. For $\lambda$ close enough to 0, this curve will also satisfy the condition 
\[ \frac{\#( \im(Y_b(\bbQ_p) \to J_b(\bbQ_p) / 2 J_b(\bbQ_p)) )}{\# J_b(\bbQ_p) /2 J_b(\bbQ_p)} \leq \frac{1}{4}. \]

Finally, we turn to Case $\mathbf{E_8}$. We now let $\cZ$ be the curve given by the equation $y^3 = (x^2 - 1)(x^3 - p^2)$, and let $\cY$ denote the projective curve over $\bbZ_p$ containing $\cZ$ and given by the multihomogeneous equation $y^3 = z (x^2 - z^2)(x^3 - p^2 z^3)$. Then $\cY$ is smooth along the unique section at infinity. We see that $\cY$ has a unique non-regular point, namely the point inside $\cZ$ corresponding to the maximal ideal $(p, x, y)$. The completed local ring of $\cZ$ at this point is isomorphic to $\bbZ_p \llbracket x, w \rrbracket / (w^3 - x^3 + p^2)$. It follows that the singularities of $\cY$ can be resolved by two blow-ups, exactly as in the ${E_6}$ case described above. Moreover, the intersection matrix is equal to $M$ as defined above, and the isomorphism class of the component group of the N\'eron model of the Jacobian of $\cY_{\bbQ_p}$ is also $(\bbZ/2\bbZ)^2$. This concludes the proof.
\end{proof}
\begin{lemma}\label{lem_large_group_of_points_modulo_2}
There exists an open subset $U_2 \subset \cB(\bbZ_2)$ such that for all $b \in U_2$, we have $\Delta(b) \neq 0$, $\cX_b(\bbZ_2) \neq \emptyset$, and the image of the map $\cX_b(\bbZ_2) \to J_b(\bbQ_2) / 2 J_b(\bbQ_2)$ does not intersect the subgroup generated by the divisor class $[(P_1) - (P_2)]$ in Case $\mathbf{E_7}$ (resp. does not contain the identity in Case $\mathbf{E_8}$).
\end{lemma}
\begin{proof}
If $c \in \cB(\bbF_2)$ is such that $\cX_{c}$ is smooth, let us write $\cY_c$ for the smooth projective completion of $\cX_{c}$ and $\cJ_c$ for $\Pic^0_{\cY_c}$. In order to prove the lemma, it suffices to exhibit a single point $c \in \cB(\bbF_2)$ such that $\cX_{c}$ is smooth, and such that the image of the map $\cX_c(\bbF_2) \to \cJ_c(\bbF_2) / 2 \cJ_c(\bbF_2)$ is nontrivial and does not intersect the given subgroup. Indeed, suppose $c$ is such a point, and define $U_2$ to be the preimage of $c$ under the natural map $\cB(\bbZ_2) \to \cB(\bbF_2)$. If $b \in U_2$, then there is a commutative diagram
\[ \xymatrix{ \cX_b(\bbZ_2) \ar[d] \ar[r] & J_b(\bbQ_2) / 2 J_b(\bbQ_2) \ar[d] \\
\cX_c(\bbF_2) \ar[r] & \cJ_c(\bbF_2) / 2 \cJ_c (\bbF_2). } \]
By Hensel's Lemma, the existence of a point in $\cX_c(\bbF_2)$ implies that $\cX_b(\bbZ_2)$ is non-empty. Since the diagram is commutative, the image of $\cX_b(\bbZ_2)$ does not intersect the subgroup generated by the divisor class $[(P_1) - (P_2)]$ in Case $\mathbf{E_7}$ (resp. does not contain the identity in Case $\mathbf{E_8}$). 

It remains to exhibit such a point $c \in \cB(\bbF_2)$ in each case. In Case $\mathbf{E_7}$, we consider the curve 
\[ \cX_c : y^3 = x^3 y + y + 1. \]
We have that $\cX_c$ is smooth over $\bbF_2$, and $\cX_c(\bbF_2)$ consists of exactly one point $(x, y) = (1, 1)$. There is an isomorphism $\cJ_c(\bbF_2) \cong \bbZ / 18 \bbZ$, hence an isomorphism $\cJ_c(\bbF_2) / 2 \cJ_c(\bbF_2) \cong \bbZ / 2 \bbZ$. The subgroup of  $\cJ_c(\bbF_2) / 2 \cJ_c(\bbF_2)$ generated by the divisor class $[(P_1) - (P_2)]$ is the trivial subgroup, while the point $(1, 1)$ has nontrivial image in $\cJ_c(\bbF_2)/2\cJ_c(\bbF_2)$ (in fact, its image in $\cJ_c(\bbF_2)$ is a generator).

In Case  $\mathbf{E_8}$, we consider the curve 
\[ \cX_c : y^3 = x^5 + y(x^3 + x^2) + x^3 + 1.\]
 We have that $\cX_c$ is smooth over $\bbF_2$, and  $\cX_c(\bbF_2)$ consists of the two points $(x, y) = (0, 1)$ and $(x, y) = (1, 1)$. There is an isomorphism $\cJ_c(\bbF_2) \cong \bbZ / 30 \bbZ$, hence an isomorphism $\cJ_c(\bbF_2) / 2 \cJ_c(\bbF_2) \cong \bbZ / 2 \bbZ$. Both of the rational points of $\cX_c(\bbF_2)$ have nontrivial image in $\cJ_c(\bbF_2) / 2 \cJ_c(\bbF_2)$.

We verified all these properties of the given curves $\cX_c$ using the \texttt{ClassGroup} functionality in \texttt{magma} \cite{Bos97}.
\end{proof}
\begin{lemma}\label{lem_curves_with_points_mod_p}
\begin{enumerate}
\item For every prime $p$, there exists an open compact subset $U_p \subset \cB(\bbZ_p)$ such that for every $b \in U_p$, $\Delta(b) \neq 0$ and $\cX_b(\bbZ_p) \neq \emptyset$.
\item There exists an integer $N_3 \geq 1$ such that for every prime $p > N_3$ and for every $b \in \cB(\bbZ_p)$ such that $\Delta(b) \neq 0$, we have $\cX_b(\bbZ_p) \neq \emptyset$.
\end{enumerate}
\end{lemma}
\begin{proof}
For each prime $p$, it is not difficult to find a point $c \in \cB(\bbF_p)$ such that $\cX_{c}$ is smooth and $\cX_c(\bbF_p)$ is non-empty. Taking $U_p$ to be the preimage of $c$ in $\cB(\bbZ_p)$ establishes the first part of the lemma. The second part follows from Hensel's Lemma and the Weil bounds; here we are implicitly using the fact, already established in the proof of Lemma \ref{lem_existence_of_local_integral_orbits}, that for any $c \in \cB(\bbF_p)$, the irreducible components of $\cX_{c}$ are geometrically irreducible.
\end{proof}
\section{Counting points}\label{sec_counting_points}

In Section \ref{sec_integral_structures} we have defined an algebraic group over $\bbZ$ and a representation $\cV$, as well as various associated structures. In Section \ref{sec_counting_points}, we continue with the same notation and now show how to estimate the number of points in $G(\bbZ) \backslash \cV(\bbZ)$ of bounded height. 

We first prove a simplified result, Theorem \ref{thm_naive_point_counting}. The more refined version (Theorem \ref{thm_refined_point_counting}), which is needed for applications, will be given at the end of this section. Let $L \subset B(\bbR)$ be one of the subsets $L_k$ described in Lemma \ref{lem_construction_of_real_sections}, and let $s : L \to V(\bbR)$ be the corresponding section. Then $L$ is a connected semialgebraic subset of $B(\bbR)$; $s$ is a semialgebraic map; and $s(L)$ has compact closure in $V(\bbR)$. The map $\Lambda \times L \to B(\bbR), (\lambda, \ell) \mapsto \lambda \cdot \ell$ given by the $\bbG_m$-action on $B$ is an open immersion, and $\Ht(\lambda \cdot \ell) = \lambda^{\deg \Delta}$. 

For any subset $A \subset \cV(\bbZ)$, we write $A^\text{irr}$ for the subset of points $a \in A$ that are $\bbQ$-irreducible, in the sense of \S \ref{section-reducible}. We recall that $r$ is the rank of $H$.
Our first result is as follows.
\begin{theorem}\label{thm_naive_point_counting}
There exist constants $C, \del > 0$, not dependent on choice of $L$, such that
\begin{equation*}
\# G(\bbZ) \backslash \{v \in [G(\bbR)\cdot \Lambda \cdot s(L)] \cap \cV(\bbZ)^{\irr} \mid \Ht(v) < \htvar\} \leq C\cdot \vol([1, \htvar^{1 / \deg \Delta}] \cdot L) + O(\htvar^{\frac{1}{2} + r / \deg \Delta - \delta}).
\end{equation*}
\end{theorem}
Our proof is very similar to that of \cite[Theorem 3.1]{Tho15}, except that a significant amount of case-by-case computation is required in order to control the contribution of elements that are `in the cusp' (i.e.\ elements that lie in the codimension-one subspace of $V$ where the coordinate corresponding to the highest root of $H$ vanishes; see Proposition \ref{prop-cuspdatum} below). To avoid repetition, we omit the details of proofs that are essentially the same as proofs appearing in \cite[\S 3]{Tho15}.

First we introduce some notation. Recall that we have fixed a choice of $\Siegel = \omega T_c K \subset G(\bbR)$ as in Proposition \ref{prop_unique_cusp}, where $\omega \subset \overline{N}(\bbR)$ is a compact subset and $T_c \subset T(\bbR)^\circ$ is open. As in \cite[Section 3.1]{Tho15}, we fix a compact semialgebraic set $G_0 \subset G(\bbR) \times \Lambda$  of non-empty interior with the property that $K \cdot G_0 = G_0$. We assume that the projection of $G_0$ to $\Lambda$ is contained in $[1, C_0]$ for some constant $C_0$ and that $\vol(G_0) = 1$. Given a subset $A \subset \cV(\bbZ)$ we let
\begin{eqnarray*}
N(A, \htvar) &=& \int_{h \in G_0} \# (\Siegel h \cdot \Lambda \cdot s(L) \cap \{v \in A^{\irr} \mid \Ht(v) < \htvar \})~ dh\\
N^*(A, \htvar)&=& \int_{h \in G_0} \# (\Siegel h \cdot \Lambda \cdot s(L) \cap \{v \in A \mid \Ht(v) < \htvar \})~ dh.
\end{eqnarray*}
The following two lemmas are the analogues in our situation of \cite[Lemma 3.3]{Tho15} and \cite[Lemma 3.4]{Tho15}; the proofs are the same.
\begin{lemma}\label{lem_N_bounds_above}
Let $A \subset \ZV(\bbZ)$ be a $G$-invariant subset. Then
\begin{equation*}
\# G(\bbZ)\setminus \{ v \in [G(\bbR) \cdot \Lambda \cdot s(L)] \cap A^{\irr} \mid \Ht(v) < \htvar \} \leq N(A, \htvar)
\end{equation*}
and
\begin{equation*}
\# G(\bbZ)\setminus \{ v \in [G(\bbR) \cdot \Lambda \cdot s(L)] \cap A \mid \Ht(v) < \htvar \} \leq N^*(A, \htvar).
\end{equation*}
\end{lemma}
\begin{lemma}\label{lem_N_bounds_below}
Given $\htvar \geq 1, n \in \overline N(\bbR), t \in T(\bbR)$, and $\lam \in \Lambda$, define $E(n, t, \lam, \htvar) = nt\lam G_0s(L) \cap \{v \in V(\bbR) \mid \Ht(v) < \htvar \}$. For any subset $A \subset \ZV(\bbZ)$, we have
\begin{equation*}
N(A, \htvar) \leq 2^{r} \int_{\lambda \in \Lambda} \int_{t \in T_c}\int_{n \in \omega}   \# [E(n, t, \lam, \htvar) \cap A^{\irr}]\delta_G(t)^{-1} dn ~dt ~d^\times \lam
\end{equation*}
and
\begin{equation*}
N^*(A, \htvar) \leq 2^{r} \int_{\lambda \in \Lambda} \int_{t \in T_c}\int_{n \in \omega} \# [E(n, t, \lam, \htvar) \cap A]\delta_G(t)^{-1} dn~ dt~ d^\times \lam,
\end{equation*}
where $\del_G$ is as defined in Section \ref{subsec_measures}.
\end{lemma}
In order to actually count points, we will use the following result, which follows from \cite[Theorem 1.3]{Bar14}. This replaces the use of \cite[Proposition 3.5]{Tho15}, itself based on a result of Davenport \cite{Dav51}. We prefer to cite \cite{Bar14} since the possibility of applying \cite{Dav51} to a general semialgebraic set rests implicitly on the Tarski--Seidenberg principle (see \cite{Dav64}). 
\begin{theorem}\label{thm_barroero}
Let $m, n \geq 1$ be integers, and let $Z \subset \bbR^{m+n}$ be a semialgebraic subset. For $T \in \bbR^m$, let $Z_T = \{ x \in \bbR^n \mid (T, x) \in Z \}$, and suppose that all such subsets $Z_T$ are bounded. Then for any unipotent upper-triangular matrix $u \in \GL_n(\bbR)$, we have
\[ \# (Z_T \cap u \bbZ^n) = \vol(Z_T) + O(\sup\{ 1, \vol(Z_{T,j}) \}), \]
where $Z_{T, j}$ runs over all orthogonal projections of $Z_T$ to any $j$-dimensional coordinate hyperplane ($1 \leq j \leq n-1)$. Moreover, the implied constant depends only on $Z$. 
\end{theorem}
To state the next proposition, we recall that for any subset $M \subset \Phi_V$, $V(M) \subset V$ is the linear subspace consisting of vectors $v = \sum_{\al \in \Phi_V} v_\alpha$ with $v_\al = 0$ for all $\al \in M$. Given disjoint subsets $M_0, M_1 \subset \Phi_V$, we define an open subscheme $V(M_0, M_1) \subset V(M_0)$ by
\[ V(M_0, M_1) = \{ v \in V(M_0) \mid v_\al \neq 0 \text{ for all } \al \in M_1 \}. \]
We also define $S(M_0) = V(M_0)(\bbQ) \cap \cV(\bbZ)$ and $S(M_0, M_1) = V(M_0, M_1)(\bbQ) \cap \cV(\bbZ)$. For ease of notation, if $M = \{\al\}$ is a single root, we write $S(M)$ as $S(\al)$.
\begin{proposition}\label{prop-cuspdatum}
Let $\al_0 \in \Phi_V$ denote the highest root of $H$ with respect to the root basis $S_H$. Then there exists $\del > 0$ such that $N( S( \al_0 ), \htvar) = O(\htvar^{\frac{1}{2} + r / \deg \Delta - \delta})$.
\end{proposition}
\begin{proof}
We call a pair $(M_0, M_1)$ of disjoint subsets of $\Phi_V$ a cusp datum. To prove the proposition, it suffices to find a set $\mathcal{C}$ of cusp data such that
\begin{enumerate}
\item $S(\al_0)^{\irr} \subset \bigcup_{(M_0, M_1) \in \mathcal{C}} S(M_0, M_1)$
\item If $(M_0, M_1) \in \mathcal{C}$, then $N^*(S(M_0, M_1), \htvar) = O(\htvar^{\frac{1}{2} + r / \deg \Delta - \delta})$.
\end{enumerate}
Consider the partial order on $\Phi_V$ given by $\be \geq \al$ if and only if $n_i(\be - \al) \geq 0$ for all $i$, where $n_i$ is as defined in Section \ref{section-reducible}. Let $\cM$ be the collection of subsets $M \subset \Phi_V$ such that if  $\al \in M$  and $\be \geq \al$  then $\be \in M$. Given a subset $M \in \cM$, we let $\lam(M) = \{ \al \in \Phi_V \mid M \cup \{ \al \} \in \cM \}$. We let $\cC$ be the collection of cusp data defined inductively as follows: in step 1, we form the cusp datum $(\{\al_0\}, \lam(\{\al_0\}))$. In each successive step  we create the set of cusp data $\{(M_0 \cup \{\al\}, \lam(M_0 \cup \{\al\})) \mid \al \in M_1\}$ for each cusp datum $(M_0, M_1)$ formed in the previous step, and then remove any cusp data such that $M_0$ satisfies any of the conditions of Proposition \ref{prop_reducibility_conditions}. By construction the collection $\cC$ satisfies condition 1 above. For each cusp datum $(M_0, M_1) \in \cC$, we check that $N^*(S(M_0, M_1), \htvar) = O(\htvar^{\frac{1}{2} + r/\deg\Delta - \del})$. To do so, by the same logic as in \cite[\S 5]{Tho15}, it suffices to find a function $f: M_1 \to \bbR_{\geq 0}$ satisfying the following two conditions:
\begin{itemize}
\item $\sum_{\al \in M_1} f(a) < \# M_0$
\item For each $1 \leq i \leq r$, we have $\sum_{\al \in \Phi_G^+} n_i(\al) - \sum_{\al \in M_0} n_i(\al) + \sum_{\al \in M_1} f(\al)n_i(\al) > 0$. 
\end{itemize}
One can program a computer to generate the list of cusp data in $\cC$, after inputting the root datum of $\frh$ and the description of its 2-grading, and then to verify that there exists such a function $f$ for each $(M_0, M_1) \in \cC$. We have carried out this verification process. Our code is available in the Mathematica notebooks \texttt{E7CuspData.nb} and \texttt{E8CuspData.nb}.\footnote{These Mathematica notebooks may be found at \url{https://www.dpmms.cam.ac.uk/~jat58/E7CuspData.nb} and \url{https://www.dpmms.cam.ac.uk/~jat58/E8CuspData.nb} respectively.} (In the name of efficiency, we actually follow a slightly different procedure, since it is time-consuming to check the condition in part 3 of Proposition \ref{prop_reducibility_conditions}. Namely, we generate a list of cusp data by eliminating only those pairs $(M_0, M_1)$ such that $M_0$ satisfies the condition in part 2 of Proposition \ref{prop_reducibility_conditions}. For the cusp data on this list, we check that either a function $f$ as above exists, or that one of the remaining conditions, i.e. part 1 or part 3 of Proposition \ref{prop_reducibility_conditions}, holds. When verifying the condition in part 3, we restrict our search to $\alpha \in M_1$. The end result is a collection of cusp data satisfying items 1 and 2 above, which suffices to prove the proposition.)
\end{proof}
\begin{proposition}\label{prop-AvN}
Let $N \geq 1$ be an integer, and let $v \in \cV(\bbZ)$. Let $A_{v, N} = v + N\cdot\cV(\bbZ)$. Then there exists $\del > 0$ such that
\begin{equation*}
N^*(A_{v, N} - S(\al_0), \htvar) \leq \frac{2^r\lvert W_0 \rvert_\infty \vol (\Siegel)}{N^{\dim V}}\vol([1, \htvar^{{1}/{\deg \Delta}}]\cdot L) + O(\htvar^{\frac{1}{2} + r / \deg \Delta - \delta}).
\end{equation*}
\end{proposition}
\begin{proof}
Let $\| \cdot \| : V(\bbR) \to \bbR_{\geq 0}$ denote the supremum norm with respect to the decomposition $\cV = \oplus_{\al \in \Phi_V} \cV_\al$ as a direct sum of free $\bbZ$-modules of rank 1. Let $J > 0$ be a constant such that $\| v \| \leq J$ for all $v \in \omega \cdot G_0 \cdot s(L)$. Let $F(n, t, \lambda, \htvar) = \{ v \in E(n, t, \lambda, \htvar) \mid | v_{\al_0} | \geq 1 \}$. If $F(n, t, \lambda, \htvar) \neq \emptyset$, then $\lambda \al_0(t) \geq 1/J$. By Theorem \ref{thm_barroero}, we have
\[ \# ((\cV(\bbZ) - S(\alpha_0)) \cap E(n, t, \lambda, \htvar)) = \# (\cV(\bbZ) \cap F(n, t, \lambda, \htvar)) = \vol(F(n, t, \lambda, \htvar)) + O(\lambda^{\dim V - 1} \al_0(t)^{-1}). \]
Similarly we have
\begin{equation} \# ((A_{v, N} - S(\alpha_0))\cap E(n, t, \lambda, \htvar)) = N^{-\dim V}\vol(F(n, t, \lambda, \htvar)) + O(\lambda^{ \dim V - 1} \al_0(t)^{-1}). 
\end{equation}
By Lemma \ref{lem_N_bounds_below}, $N^\ast(A_{v, N} - S(\al_0), \htvar) $ is bounded above by 
\begin{equation}\label{eqn_N_ast_estimate} \begin{split} 2^r  \int_{\lambda \in \Lambda} \int_{t \in T_c} \int_{n \in \omega} N^{-\dim V}  \vol(F(n, t, \lambda, \htvar)) \delta_G(t)^{-1} \, dn \, dt \, d^\times \lambda \\ + 2^r \int_{\lambda = C_0^{-1}}^{\htvar^{1 / \deg \Delta}} \int_{t \in T_c} \int_{n \in \omega} O(\lambda^{\dim V - 1} \al_0(t)^{-1}) \delta_G(t)^{-1} \, dn \, dt \,d^\times \lambda.\end{split}
\end{equation}
The second term in (\ref{eqn_N_ast_estimate}) is $O(\htvar^{\frac{1}{2} + (r-1) / \deg \Delta})$. Lemma \ref{lem_construction_of_real_sections} shows that the first term is bounded above by
\[ \begin{split} &2^r\int_{\lambda \in \Lambda} \int_{t \in T_c} \int_{n \in \omega} N^{-\dim V}  \vol(E(n, t, \lambda, \htvar)) \delta_G(t)^{-1} \, dn \, dt \, d^\times \lambda \\ & = \frac{2^r}{N^{\dim V} }\int_{\lambda \in \Lambda} \int_{g \in \Siegel}  \int_{v \in V(\bbR)} \int_{h \in G_0} \mathbf{1}_{v \in g h \lambda s(L), \Ht(v) < \htvar} \, dh \, dv \, dg \, d^\times \lambda \\
& \leq \frac{2^r| W_0 |_\infty}{N^{\dim V}} \int_{h \in G_0} \vol(\Siegel) \vol([1, \htvar^{1 / \deg \Delta}] \cdot L) \, dh. \\
& = \frac{2^r| W_0 |_\infty}{N^{\dim V}} \vol(\Siegel) \vol([1, \htvar^{1 / \deg \Delta}] \cdot L).
\end{split} \]
This completes the proof. 
\end{proof}
We can now finish the proof of Theorem \ref{thm_naive_point_counting}. By Lemma \ref{lem_N_bounds_above}, we have 
\[ G(\bbZ)\backslash\{v \in [G(\bbR)\cdot \Lambda \cdot s(L)] \cap \cV(\bbZ)^{\irr} \mid \Ht(v) < \htvar\} \leq N(\cV(\bbZ), \htvar) \leq N(\cV(\bbZ) - S(\al_0), \htvar) + N^\ast(S(\al_0), \htvar). \]
The result now follows on combining Proposition \ref{prop-cuspdatum} and Proposition \ref{prop-AvN}.

We now state the more refined version of Theorem \ref{thm_naive_point_counting} mentioned at the beginning of this section.
\begin{theorem}\label{thm_refined_point_counting}
Let $p_1, \dots, p_s$ be distinct primes, and for each $i \in \{1, \dots, s\}$, let $V_{p_i} \subset \cV(\bbZ_{p_i}) \cap V(\bbQ_{p_i})^{\text{reg.ss.}}$ be an open compact subset that is $G(\bbQ_{p_i})$-invariant, in the sense that if $v \in V_{p_i}$, $g \in G(\bbQ_{p_i})$ and $gv \in \cV(\bbZ_{p_i})$, then $gv \in V_{p_i}$. Let $A = \cV(\bbZ) \cap (V_{p_1} \times \dots \times V_{p_s})$ (where we are identifying $\cV(\bbZ)$ with its image in $\cV(\bbZ_{p_1}) \times \dots \times \cV(\bbZ_{p_s})$ under the diagonal embedding). Then there exist constants $C, \delta > 0$ not depending on $s$ or the sets $V_{p_1}, \dots, V_{p_s}$ such that
\[ \# G(\bbQ) \backslash \{ v \in A^{\irr} \mid \Ht(v) < \htvar \} \leq C \left( \prod_{i=1}^s \int_{b \in \cB(\bbZ_{p_i})} \frac{ \# ( G(\bbQ_{p_i}) \backslash (V_{p_i} \cap V_b(\bbQ_{p_i})) ) }{ \# \Stab_{G(\bbQ_{p_i})}(\kappa_b)} \, db \right) \htvar^{\frac{1}{2} + r /\deg \Delta} + O(\htvar^{\frac{1}{2} + r /\deg \Delta - \delta}). \]
\end{theorem}
\begin{proof}
We recall that for each prime $p$ we have defined in the statement of Proposition \ref{prop_description_of_measures_padic_case} a locally constant function $m_p : V(\bbQ_p)^\text{reg.ss.} \to \bbR$ by the formula
\[ m_p(v) = \sum_{v' \in G(\bbZ_p) \backslash ( G(\bbQ_p) \cdot v \cap \cV(\bbZ_p) ) } \frac{ \# \Stab_{G(\bbQ_p)}(v) }{ \# \Stab_{G(\bbZ_p)}(v') }. \]
 The same argument as in the proof of \cite[Corollary 3.9]{Tho15} leads to an estimate
\[ \# G(\bbQ) \backslash \{ v \in A^\text{irr} \mid \Ht(v) < \htvar \} \leq 2^r \sum_{\substack{v \in G(\bbZ) \backslash A \\ \Ht(v) < \htvar}} \frac{1}{m_{p_i}(v)}. \]
Combining Lemma \ref{lem_N_bounds_above}, Proposition \ref{prop-AvN}, and Proposition \ref{prop-cuspdatum}, and summing over all choices of $L$ as in Lemma \ref{lem_construction_of_real_sections}, yields absolute constants $C, \delta > 0 $ such that
\[ \sum_{\substack{v \in G(\bbZ) \backslash A \\ \Ht(v) < \htvar}} \frac{1}{m_{p_i}(v)} \leq C \left( \prod_{i=1}^s \int_{v \in V_{p_i}} \frac{1}{m_{p_i}(v)} \, dv \right) \htvar^{\frac{1}{2} + r/\deg \Delta} + O(\htvar^{\frac{1}{2} + r/\deg \Delta - \delta}). \]
By the third part of Proposition \ref{prop_description_of_measures_padic_case}, this expression is equal to 
\[ C \left( \prod_{i=1}^s | W_0 |_{p_i} \vol(G(\bbZ_{p_i})) \right) \left( \prod_{i=1}^s  \int_{b \in \cB(\bbZ_{p_i})} \frac{ \# ( G(\bbQ_{p_i}) \backslash (V_{p_i} \cap V_b(\bbQ_{p_i})) ) }{ \# \Stab_{G(\bbQ_{p_i})}(\kappa_b)} \, db \right) \htvar^{\frac{1}{2} + r/\deg \Delta} + O(\htvar^{\frac{1}{2} + r/\deg \Delta - \delta}). \]
The products $ \prod_{i=1}^s | W_0 |_{p_i} \vol(G(\bbZ_{p_i}))$ can be bounded independently of $s$ and the primes $p_1, \dots, p_s$. They can therefore be absorbed into the constant, giving the estimate in the statement of the theorem.
\end{proof}

\section{Applications to 2-Selmer sets}

In this final section, we prove our main theorems, including the results stated in the introduction, by combining all the theory developed so far. In order to avoid confusion, we treat each of the two families of curves (corresponding to Case $\mathbf{E_7}$ and Case $\mathbf{E_8}$) in turn.
\subsection{Applications in Case $\mathbf{E_7}$}\label{sec_applications_to_E7_case}
As above, we write $\cB = \Spec \bbZ[\p_2, \p_6, \p_8, \p_{10}, \p_{12}, \p_{14}, \p_{18}]$ for affine space over $\bbZ$ in 7 variables, and write $\cX \to \cB$ for the family of affine plane curves given by equation (\ref{eqn_intro_curves_of_type_E7}):
\begin{equation*}
y^3 = x^3 y + \p_{10} x^2 + x(\p_2 y^2 + \p_8 y + \p_{14} ) + \p_6 y^2 + \p_{12} y + \p_{18}.
\end{equation*}
This family has the following interpretation:
\begin{proposition}\label{prop_universality_of_E7_family} Let $k / \bbQ$ be a field. Then:
\begin{enumerate}
\item The locus inside $\cB_k$ above which the morphism $\cX_k \to \cB_k$ is smooth is the complement of an irreducible closed subset of $\cB_k$ of codimension 1.
\item The set of points $b \in \cB(k)$ for which $\cX_b$ is smooth is in bijection with the set of equivalence classes of triples $(C, P_1, t)$, where:
\begin{enumerate}
\item $C$ is a smooth, non-hyperelliptic curve of genus 3 over $k$.
\item $P_1 \in C(k)$ is a flex point in the canonical embedding, i.e.\ the projective tangent line to $C$ at $P_1$ intersects $C$ with multiplicity 3 at the point $P_1$.
\item $t \in T_{P_1} C$ is a non-zero Zariski tangent vector at the point $P_1$.
\end{enumerate}
If $b$ corresponds to $(C, P_1, t)$, then $\cX_b$ is isomorphic to $C - \{ P_1, P_2 \}$, where $P_2 \in C(k)$ is the unique point such that $3 P_1 + P_2$ is a canonical divisor. For $\lambda \in k^\times$, the coefficients $\p_i$ satisfy the equality
\[ \p_i (C, P_1, \lambda t) = \lambda^{i/2} \p_i(C, P_1, t). \]
\end{enumerate}
\end{proposition}
\begin{proof}
Part 1 follows from the fact that $\cX_b$ is smooth if and only if $\Delta(b) \neq 0$.
The proof of the second part is very similar to the proof of \cite[Lemma 4.1]{Tho15}, although here we cannot appeal to Pinkham's Theorem. Let $(C, P_1, t)$ be a tuple of the type described in the proposition, and let $P_2 \in C(k)$ be the point such that $3 P_1 + P_2$ is a canonical divisor. The Riemann--Roch Theorem shows that $h^0(C, \cO_C(3P_1)) = 2$ and $h^0(C, \cO_C(2P_1 + P_2)) = 2$. We can therefore find functions $y, x \in k(C)^\times$, uniquely determined up to addition of constants, such that the polar divisor of $y$ is $3 P_1$ and the polar divisor of $x$ is $2 P_1 + P_2$, and such that $y = z^{-3} + \dots$, $x = z^{-2} + \dots$ locally at the point $P_1$, where $z$ is a local parameter at $P_1$ such that $dz(t) = 1$. We can also assume that $y$ vanishes at the point $P_2$.

The 10 monomials 
\[ 1, x, x^2, y, y x, yx^2, yx^3, y^2, y^2 x, y^3  \]
all lie in the 9-dimensional space $H^0(C, \cO_C(9P_1 + 2 P_2))$ and are linearly independent, as can be seen by considering their polar divisors. It follows that they satisfy a unique linear relation of the form
\begin{equation}\label{eqn_e7_plane_embedding} y^3 = x^3 y +  x^2 ( \p_4 y + \p_{10} ) + x(\p_2 y^2 + \p_8 y + \p_{14}) + \p_6 y^2 + \p_{12} y + \p_{18}. 
\end{equation}
The function $y$ is uniquely determined by the above data. We also see that there is a unique translate $x + a$ ($a \in k$) such that, after replacing $x$ by $x + a$, we have $\p_4 = 0$ in equation (\ref{eqn_e7_plane_embedding}). The homogenization of the equation  (\ref{eqn_e7_plane_embedding}) then describes the canonical embedding of the curve $C$. 
\end{proof}
If $k / \bbQ$ is a field extension and $b \in \cB(k)$ is such that $\cX_b$ is smooth, then we write $Y_b$ for the unique smooth projective completion of $\cX_b$.

As in the introduction, we define $\cF_0 = \{ b \in \cB(\bbZ) \mid \cX_{b, \bbQ} \text{ is smooth} \}$.
We say that a subset $\cF \subset \cF_0$ is defined by congruence conditions if there exist distinct primes $p_1, \dots, p_s$ and a non-empty open compact subset $U_{p_i} \subset \cB(\bbZ_{p_i})$ for each $i \in \{1, \dots, s\}$ such that
\begin{equation*} \cF = \cF_0 \cap (U_{p_1} \times \dots \times U_{p_s}), 
\end{equation*}
where we are taking the intersection inside $\cB(\bbZ_{p_1}) \times \dots \times \cB(\bbZ_{p_s})$.

We recall that for $b \in \cB(\bbR)$ we have defined $\Ht(b) = \sup_i | \p_i(b) |^{{126} / i}$. This function is homogeneous of degree $126$, in the sense that for $\lambda \in \bbR^\times$, we have $\Ht(\lambda \cdot b) = | \lambda |^{126} \Ht(b)$. (We note that 126 is the number of roots in the root system of type $E_7$, and so also the degree of the discrimimant polynomial $\Delta$ considered in \S \ref{sec_vinberg_representations}.)
\begin{lemma}\label{lem_curves_of_bounded_height} There exists a constant $\delta > 0$ such that if $\cF \subset \cF_0$ is a subset defined by congruence conditions as above, then 
\[ \# \{ b \in \cF \mid \Ht(b) < \htvar \} = \left( \prod_{i=1}^r \vol(U_{p_i}) \right) \htvar^{\frac{1}{2} + \frac{7}{126}} + O(\htvar^{\frac{1}{2} + \frac{7}{126} - \delta}) \]
as $\htvar \to \infty$.
\end{lemma}
\begin{proof}
This is an easy consequence of Theorem \ref{thm_barroero}.
\end{proof}

Our main theorems are now as follows.
\begin{theorem}\label{thm_E7_application_1}
Let $\cF \subset \cF_0$ be a subset defined by congruence conditions. Then
\[ \limsup_{\htvar \to \infty} \frac{ \sum_{\substack{b \in \cF \\ \Ht(b) < \htvar}} \# \Sel_2(Y_b) }{\# \{ b \in \cF \mid \Ht(b) < \htvar \} } < \infty. \]
\end{theorem}
In order to state the next theorem, we observe that if $b \in \cB(\bbQ)$ is such that $\cX_b$ is smooth, then the 2-Selmer set $\Sel_2(Y_b)$ always contains the `trivial' classes arising from divisors supported on the points $P_1$, $P_2$ at infinity (as in the statement of Proposition \ref{prop_universality_of_E7_family}). We write $\Sel_2(Y_b)^{\text{triv}}$ for the subset of $\Sel_2(Y_b)$ consisting of these classes, and note that $\#\Sel_2(Y_b)^{\text{triv}} \leq 2$, with equality if and only if the divisor class $[(P_2) - (P_1)]$ is not divisible by 2 in $J_b(\bbQ)$.
\begin{theorem}\label{thm_E7_application_2}
For any $\epsilon > 0$, there exists a subset $\cF \subset \cF_0$ defined by congruence conditions such that 
\[ \limsup_{\htvar \to \infty} \frac{ \sum_{\substack{b \in \cF \\ \Ht(b) < \htvar}} \# \Sel_2(Y_b) }{\# \{ b \in \cF \mid \Ht(b) < \htvar \} } < 2 + \epsilon. \]
Consequently, for any such choice of $\cF$ we have
\[ \liminf_{\htvar \to \infty} \frac{ \# \{ b \in \cF \mid \Ht(b) < \htvar \text{ \emph{and} } \Sel_2(Y_b) = \Sel_2(Y_b)^{\text{\emph{triv}}} \} }{ \# \{ b \in \cF \mid \Ht(b) < \htvar \} } > 1 - \epsilon. \]
\end{theorem}
The proof of Theorem \ref{thm_E7_application_2} is essentially a refined version of the proof of Theorem \ref{thm_E7_application_1}, so we just give the proof of Theorem \ref{thm_E7_application_2}.
\begin{proof}[Proof of Theorem \ref{thm_E7_application_2}]
Let $p_1, \dots, p_s$ be primes congruent to 1 modulo 6. Let $\varep \in (0, 1)$ be as in Lemma \ref{lem_large_neron_component_group}, and for each $i \in \{1, \dots, s\}$, let $U_{p_i} \subset \cB(\bbZ_{p_i})$ be the set described in the statement of Lemma \ref{lem_large_neron_component_group}. These sets have the following property: define 
\[ V_{p_i} = \pi^{-1}(U_{p_i}) \cap  \cV(\bbZ_{p_i}) \cap ([G(\bbQ_{p_i})\cdot X(\bbQ_{p_i})] \cup [G(\bbQ_{p_i})\cdot \kappa(\bbQ_p)] \cup [G(\bbQ_{p_i})\cdot \kappa'(\bbQ_p)]) , \] where $\kappa'$ is any Kostant section that is not $G$-conjugate to $\kappa$.
Then $V_{p_i}$ is an open compact subset of $\cV(\bbZ_{p_i})^\text{reg.ss.}$, and for any $b \in U_{p_i}$ we have $\Delta(b) \neq 0$ and
\begin{equation}\label{eqn-orbit-ineq}
\frac{\#(G(\bbQ_{p_i})\backslash (V_{p_i} \cap V_b(\bbQ_{p_i}))}{\# \Stab_{G(\bbQ_{p_i})}(\kappa_b)} \leq \varep.
\end{equation}
We let $\cF = \cF_0 \cap (U_{p_1} \times \dots \times U_{p_s})$. For any $b \in \cF$, let $\Sel_2(Y_b)^\text{irr} \subset \Sel_2(Y_b)$ denote the subset of `nontrivial' elements, i.e.\ the complement of $\Sel_2(Y_b)^{\text{triv}}$ in $\Sel_2(Y_b)$.
Let $A = \cV(\bbZ) \cap (V_{p_1} \times \dots \times V_{p_s})$. Then by Proposition \ref{prop_existence_of_global_integral_orbits}, for any $\htvar > 0$ we have
\[ \sum_{\substack{b \in \cF \\ \Ht(b) < \htvar}} \# \Sel_2(Y_b)^\text{irr} \leq G(\bbQ) \backslash \{ v \in A^\text{irr} \mid \Ht(v) < N_1^{\deg \Delta} \htvar \}.\]
By combining Theorem \ref{thm_refined_point_counting}, Lemma \ref{lem_curves_of_bounded_height}, and the inequality (\ref{eqn-orbit-ineq}), we see that there exist constants $C, \delta > 0$, not depending on $s$ or the choice of primes $p_1, \dots, p_s$, such that
\[ \frac{\sum_{\substack{b \in \cF \\ \Ht(b) < \htvar}} \# \Sel_2(Y_b)^\text{irr}}{\# \{ b \in \cF \mid \Ht(b) < \htvar \}} \leq \frac{\varep^s C + O(\htvar^{- \delta})}{1 + O(\htvar^{ - \delta})}. \]
Since $\# \Sel_2(Y_b) \leq 2 + \# \Sel_2(Y_b)^\text{irr}$, the first sentence in the statement of the theorem now follows on choosing $s$ sufficiently large and letting $\htvar \to \infty$. The second sentence follows from the first on combining it with the following lemma.
\end{proof}
\begin{lemma}
Let $\cF \subset \cF_0$ be a family defined by congruence conditions. Then the limit
\[ \lim_{\htvar \to \infty} \frac{ \# \{ b \in \cF \mid \Ht(b) < \htvar, \# \Sel_2(Y_b)^\text{\emph{triv}} = 2 \} }{ \# \{ b \in \cF \mid \Ht(b) < \htvar \} } \]
exists and equals 1.
\end{lemma}
\begin{proof}
Let $b \in \cF$, and let $C_b = Z_H(\kappa_b)$, a maximal torus of $H$. The Galois action on $C_b$ induces an associated homomorphism $\Gal(\bbQ^s / \bbQ) \to W(H, C_b)$. Corollary \ref{cor_generic_non-triviality_of_trivial_divisor_class} shows that if this homomorphism is surjective, then $\# \Sel_2(Y_b)^\text{triv} = 2$. It therefore suffices to show that the limit
\[ \lim_{\htvar \to \infty} \frac{ \# \{ b \in \cF \mid \Ht(b) < \htvar,\Gal(\bbQ^s / \bbQ) \to W(H, C_b) \text{ surjective} \} }{ \# \{ b \in \cF \mid \Ht(b) < \htvar \} } \]
exists and equals 1. This is a variant of the Hilbert Irreducibility Theorem and can be proved along similar lines to the arguments in \cite[\S 13.2]{Ser97}.
\end{proof}
\begin{theorem}\label{thm_E7_application_3}
For any $\epsilon > 0$, there exists a subset $\cF \subset \cF_0$ defined by congruence conditions such that the following conditions are satisfied:
\begin{enumerate}
\item For every $b \in \cF$ and every prime $p$, we have $\cX_b(\bbZ_p) \neq \emptyset$.
\item We have 
\[ \liminf_{\htvar \to \infty} \frac{ \# \{ b \in \cF \mid \cX_b(\bbZ_{(2)}) = \emptyset \} } { \# \{ b \in \cF \mid \Ht(b) < \htvar \} } > 1 - \epsilon. \]
\end{enumerate}
\end{theorem}
For the sets $\cF$ constructed in Theorem \ref{thm_E7_application_3}, we may say that a positive proportion of the curves $\cX_b$ ($b \in \cF$) have integral points everywhere locally, but no integral points globally.
\begin{proof}
By Lemma \ref{lem_large_group_of_points_modulo_2} and Lemma \ref{lem_curves_with_points_mod_p}, we can choose for every prime $p$ an open compact subset $U_p \subset \cB(\bbZ_p)$ such that the following conditions are satisfied:
\begin{enumerate}
\item For each $b \in U_2$, $\Delta(b) \neq 0$ and the image of the map $\cX_b(\bbZ_2) \to J_b(\bbQ_2) / 2 J_b(\bbQ_2)$ does not intersect the subgroup generated by $[(P_1) - (P_2)]$.
\item For every prime $p$ and for every $b \in U_p$ such that $\Delta(b) \neq 0$, the set $\cX_b(\bbZ_p)$ is non-empty.
\item For every sufficiently large prime $p$, $U_p = \cB(\bbZ_p)$.
\end{enumerate}
Let $\cF \subset \cF_0$ be the corresponding subset defined by congruence conditions. Fix $\epsilon > 0$. By modifying $U_p$ at sufficiently many primes congruent to $1 \text{ modulo }6$, as in the proof of Theorem \ref{thm_E7_application_2}, we can assume moreover that the following condition is satisfied:
\begin{enumerate}
\item[4.] We have
\[ \liminf_{X \to \infty} \frac{ \# \{ b \in \cF \mid \Ht(b) < \htvar \text{ and } \Sel_2(Y_b) = \Sel_2(Y_b)^{\text{triv}}\}}{\# \{ b \in \cF \mid \Ht(b) < \htvar \}} > 1 - \epsilon. \]
\end{enumerate}
To complete the proof of the theorem, we just need to show that if $b \in \cF$ is such that $\Sel_2(Y_b) = \Sel_2(Y_b)^{\text{triv}}$, then $\cX(\bbZ_{(2)}) = \emptyset$. To this end, we consider the commutative diagram
\[ \xymatrix{ \cX_b(\bbZ_{(2)}) \ar[r] \ar[d] & \cX_b(\bbZ_2) \ar[d] \\
\Sel_2(Y_b) \ar[r] & J_b(\bbQ_2) / 2 J_b(\bbQ_2), } \]
where the maps are the natural ones. By construction of $U_2$, the image of the right-hand vertical map is contained in the  complement of the subgroup generated by the divisor class $[(P_1) - (P_2)]$. By assumption, the image of the bottom horizontal map is contained in  the subgroup generated by the divisor class $[(P_1) - (P_2)]$. This forces $\cX_b(\bbZ_{(2)})$ to be empty, as desired.
\end{proof}

\subsection{Applications in Case $\mathbf{E_8}$}\label{sec_applications_to_E8_case}
We now forget the notation of \S \ref{sec_applications_to_E7_case}, and write $\cB = \Spec \bbZ[\p_2, \p_8, \p_{12}, \p_{14}, \p_{18}, \p_{20}, \p_{24}, \p_{30}]$ for affine space over $\bbZ$ in 8 variables, and write $\cX \to \cB$ for the family of affine plane curves given by equation (\ref{eqn_intro_curves_of_type_E8}):
\begin{equation*}
y^3 = x^5 + y(\p_2 x^3 + \p_8 x^2 + \p_{14} x + \p_{20} ) + \p_{12} x^3 + \p_{18} x^2 + \p_{24} x + \p_{30}.
\end{equation*}
This family has the following interpretation:
\begin{proposition}\label{prop_universality_of_E8_family} Let $k / \bbQ$ be a field. Then:
\begin{enumerate}
\item The locus inside $\cB_k$ above which the morphism $\cX_k \to \cB_k$ is smooth is the complement of an irreducible closed subset of $\cB_k$ of codimension 1.
\item The set of points $b \in \cB(k)$ for which $\cX_b$ is smooth is in bijection with the set of equivalence classes of triples $(C, P, t)$, where:
\begin{enumerate}
\item $C$ is a smooth, non-hyperelliptic curve of genus 4 over $k$.
\item $P \in C(k)$ is a point such that $6P$ is a canonical divisor and $h^0(C, \cO_C(3P)) = 2$.
\item $t \in T_{P} C$ is a non-zero Zariski tangent vector at the point $P$.
\end{enumerate}
If $b$ corresponds to $(C, P_1, t)$, then $\cX_b$ is isomorphic to $C - \{ P \}$. For $\lambda \in k^\times$, the coefficients $\p_i$ satisfy the equality
\[ \p_i (C, P, \lambda t) = \lambda^i \p_i(C, P, t). \]
\end{enumerate}
\end{proposition}
The proof is very similar to the proof of \cite[Lemma 4.1]{Tho15} and to the proof of Proposition \ref{prop_universality_of_E7_family}, so we omit it.

If $k / \bbQ$ is a field extension and $b \in \cB(k)$ is such that $\cX_b$ is smooth, then we write $Y_b$ for the unique smooth projective completion of $\cX_b$. As in Case $\mathbf{E_7}$, we define $\cF_0 = \{ b \in \cB(\bbZ) \mid \cX_{b, \bbQ} \text{ is smooth} \}$, and we say that a subset $\cF \subset \cF_0$ is defined by congruence conditions if there exist distinct primes $p_1, \dots, p_s$ and a non-empty open compact subset $U_{p_i} \subset \cB(\bbZ_{p_i})$ for each $i \in \{1, \dots, s\}$ such that
\[\cF = \cF_0 \cap (U_{p_1} \times \dots \times U_{p_s}). \]

If $b \in \cB(\bbR)$, then we have $\Ht(b) = \sup_i | \p_i(b) |^{240 / i}$. This function is homogeneous of degree $240$, in the sense that for $\lambda \in \bbR^\times$, we have $\Ht (\lambda b) = | \lambda |^{240} \Ht(b)$.
As in Case $\mathbf{E_7}$, an application of Theorem \ref{thm_barroero} shows that there exists a constant $\delta > 0$ such that if $\cF \subset \cF_0$ is a subset defined by congruence conditions as above, then
\[ \# \{ b \in \cF \mid \Ht(b) < \htvar \} = \left( \prod_{i=1}^s \vol(U_{p_i}) \right) \htvar^{\frac{1}{2} + \frac{1}{30}} + O(\htvar^{\frac{1}{2} + \frac{1}{30} - \delta}) \]
as $\htvar \to \infty$.

Our main theorems in Case $\mathbf{E_8}$ are as follows. We omit the proofs since they are similar, and simpler, than those in Case  $\mathbf{E_7}$ in the previous section.
\begin{theorem}\label{thm_E8_application_1}
Let $\cF \subset \cF_0$ be a subset defined by congruence conditions. Then
\[ \limsup_{\htvar \to \infty} \frac{ \sum_{\substack{b \in \cF \\ \Ht(b) < \htvar}} \# \Sel_2(Y_b) }{ \# \{ b \in \cF \mid \Ht(b) < \htvar \}  } < \infty. \]
\end{theorem}

\begin{theorem}\label{thm_E8_application_2}
For any $\epsilon > 0$, there exists a subset $\cF \subset \cF_0$ defined by congruence conditions such that 
\[ \limsup_{\htvar \to \infty} \frac{ \sum_{\substack{b \in \cF \\ \Ht(b) < \htvar}} \# \Sel_2(Y_b) }{ \# \{ b \in \cF \mid \Ht(b) < \htvar \} } < 1 + \epsilon. \]
Consequently, we have
\[ \liminf_{\htvar \to \infty} \frac{ \# \{ b \in \cF \mid \Ht(b) < \htvar \text{ and }\# \Sel_2(Y_b)= 1 \} }{ \# \{ b \in \cF \mid \Ht(b) < \htvar \} } > 1 - \epsilon. \]
\end{theorem}

\begin{theorem}\label{thm_E8_application_3}
For any $\epsilon > 0$, there exists a subset $\cF \subset \cF_0$ defined by congruence conditions such that the following conditions are satisfied:
\begin{enumerate}
\item For every $b \in \cF$ and every prime $p$, we have $\cX_b(\bbZ_p) \neq \emptyset$.
\item We have 
\[ \liminf_{\htvar \to \infty} \frac{ \# \{ b \in \cF \mid \cX_b(\bbZ_{(2)}) = \emptyset \} } { \# \{ b \in \cF \mid \Ht(b) < \htvar \} } > 1 - \epsilon. \]
\end{enumerate}
\end{theorem}

\bibliographystyle{alpha}
\bibliography{bounded_selmer_E7_E8}
\end{document}